\documentclass[12pt]{article}
{\begin{Sbox}\begin{minipage}}%
{\end{minipage}\end{Sbox}\fbox{\TheSbox}}%

\usepackage[psamsfonts]{amssymb}
\usepackage{amsmath, amsthm, anysize, enumerate, color, url}
\usepackage{graphicx}
\usepackage{epstopdf}
\usepackage{epsfig}
\usepackage{subfigure}
\usepackage[ruled,linesnumbered]{algorithm2e}

\usepackage[sectionbib]{natbib}

\usepackage{threeparttable}

\newtheorem{theorem}{Theorem} 

\newtheorem{lemma}{Lemma} 

\usepackage{xr-hyper}  
\usepackage[colorlinks, breaklinks,linkcolor=red,citecolor=blue]{hyperref}

\usepackage{breakurl}

\newtheorem{definition}{Definition}

\theoremstyle{definition}
\newtheorem{remark}{Remark}  

\newcommand{\E}{\mathbb{E}}

\newcommand{\R}{\mathbb{R}}

\renewcommand{\P}{\mathbb{P}}

\usepackage{setspace}

\usepackage{float}
\usepackage{multirow}

\usepackage{booktabs} 
\usepackage{caption}

\begin{document}

\title{Inference in the $p_0$ model for directed networks under local differential privacy}

\author{
Xueying Sun\footnote{Department of Statistics, Central China Normal University, Wuhan, 430079, China. Email: sxy@mails.ccnu.edu.cn}
\and Ting Yan\footnote{Department of Statistics, Central China Normal University, Wuhan, 430079, China. Email: tingyanty@mail.ccnu.edu.cn. }
\and Binyan Jiang\footnote{Department of Data Science and Artificial Intelligence, Hong Kong Polytechnic University, Hong Kong, China. Email: by.jiang@polyu.edu.hk
}
}

\date{}

\maketitle

\begin{abstract}

We explore the edge-flipping mechanism, a type of input perturbation, to release the directed graph under edge-local differential privacy.
By using the noisy bi-degree sequence from the output graph, we construct the moment equations to estimate the unknown parameters  in the $p_0$ model,
which is an exponential family distribution with the bi-degree sequence as the natural sufficient statistic. 
We show that the resulting private estimator is asymptotically consistent and normally distributed under some conditions.
In addition, we compare the performance of input and output perturbation mechanisms for releasing bi-degree sequences 
in terms of parameter estimation accuracy and privacy protection.
Numerical studies demonstrate our theoretical findings and compare the performance of the private estimates 
obtained by different types of perturbation methods. We apply the proposed method to analyze the UC Irvine message network.

\end{abstract}

\section{Introduction}\label{section-introduction}

As increasing amounts of network data (of all kinds, but especially social data) have been collected and made publicly available, privacy has become an important issue in network data analysis because network data often contain sensitive information about individuals and their relationships (e.g., sexual relationships, email exchanges).
However, ``naive" privacy-preserving methods, such as anonymization (removing the basic identifiers such as name, social security number, etc.) have been shown to fail and can lead to disclosure of individual relationships or characteristics associated with the released network [ \cite{Narayanan:Shmatikov:2009}, \cite{Backstrom:Dwork:Kleinberg:2011}].
A formal standard for data privacy is known as \emph{differential privacy} (DP) for randomized data-releasing mechanisms to provably limit the worst case disclosure risk in the presence of any arbitrary external information [\cite{Dwork:Mcsherry:Nissim:Smith:2006}].
A DP algorithm requires that the outputs should not be significantly different if the inputs are similar.

According to the underlying architecture, DP can be divided into two types: \emph{centralized DP} and \emph{local DP} (LDP) [\cite{Duchi:2013}].
DP was originally designed for a \emph{centralized} setting, in which a trusted data curator processes a database with the exact data records of multiple users, and publishes perturbed statistics or other data analysis results from the database using a randomized mechanism [\cite{Qin:2017}].
Note that the existing privacy protection and statistical analysis methods under \emph{centralized} DP cannot be applied to decentralized graphs, since the data cannot be centrally collected in the first place [\cite{Qin:2017}].
In the \emph{local} setting, which is the focus of this paper, the data curator is not trusted, instead, each user perturbs her data locally with a differentially private mechanism such that sensitive information of individuals is kept private even from data collection process [\cite{Imola:2021}], which achieves more thorough privacy protection.
At present, LDP technology has been applied in the industrial field: Apple has applied this technology to the operating system iOS 10 to protect users' device data, and Google also uses this technology to collect users' behavior statistics from the Chrome browser [\cite{Bebensee:2019,Erlingsson:2014}].
One example is to use a LDP algorithm to detect frequently used emojis while maintaining user's privacy [\cite{Bebensee:2019}].

Algorithms have been developed for releasing network data or their summary statistics safely under \emph{centralized} DP [\cite{Hay:2009, Lu:Miklau:2014,Nguyen:Imine:Rusinowitch:2016,Karwa:Slakovic:2016,Yan:2021}]. 
Under \emph{local} DP framework, \cite{Qin:2017} investigated techniques to ensure LDP for individuals while collecting structural information and generating representative synthetic social graphs. \cite{Karwa:2017} proposed a randomized response mechanism to generate synthetic networks, and then used
Markov chain Monte Carlo techniques to fit exponential-family random graph models.
\cite{Hehir:2022} used a LDP randomized response mechanism called the edge-flipping mechanism
and developed theoretical guarantees for differentially private community detection.
\cite{Chang:2024} conducted an
in-depth study of the trade-off between the level of privacy and the efficiency of statistical inference for parameter estimation in the $\beta$-model
with the randomized response mechanism as in \cite{Karwa:2017} under the edge-differential privacy.
However, statistical inference based on LDP bi-degree sequences in directed networks have not been explored.

In this paper, we conduct statistical inference for parameter estimation in the $p_0$ model by the edge-flipping mechanism under edge-LDP.
The $p_0$ model is an exponential random graph model with the bi-degree sequence as its exclusively sufficient statistic.
Note that the released sequences of many output perturbation mechanisms, such as the classical Laplace mechanism, are usually not graphic degree sequences [\cite{Yan:2021}], which may lead to invalid inference. The local private edge-flipping mechanism avoids this problem.
We use the edge-flipping mechanism under edge-LDP to release a sensitive graph represented as the adjacency matrix.
The main contributions are as follows.
(1) To the best of our knowledge, for the first time under LDP, we conduct parameter estimation based on the moment equation using the flipped bi-degree sequence and prove consistency and asymptotic normality of private estimators in the $p_0$ model.
The asymptotic inference in this model is nonstandard since the number of parameters is twice that of nodes and increases as the size of the network grows.
(2) We further compare the performance of the estimations in $p_0$ model via the Laplace mechanism (output perturbation) under \emph{centralized} DP and the edge-flipping mechanism (input perturbation) under LDP and demonstrate their trade-off between information publication and inference accuracy.
It must be emphasized that in the case where centralized DP is unavailable, the synthetic graph released through the edge-flipping mechanism under LDP can still effectively estimate the parameters of the $ p_0$ model. Understandably, the edge-flipping mechanism adds more noise because more information even the whole network is published, thus it will lose some accuracy of the estimation but still maintains consistency and asymptotic normality. It is worth noting that the edge-flipping mechanism can be used to perform more downstream tasks, not just the parameter estimation [\cite{Hehir:2022}].
(3) We provide simulation studies and an analysis of the UC Irvine message networks to confirm our theoretical claims.

The remainder of the paper is organized as follows.
In Section \ref{section-preliminaries}, we introduce the necessary background on \emph{centralized} and \emph{local} DP and the $p_0$ model used in this paper. Moreover, considering both the out-edge and the in-edge simultaneously,
we present some definitions of edge-LDP for directed networks and the corresponding input perturbation mechanisms. In Section \ref{section-edge flipping}, we use the edge-flipping mechanism to release directed graphs.
In Section \ref{section-theoretical properties},
we present the estimation in the $p_0$ model using the edge-\emph{local} DP bi-degree sequence and show the consistency and asymptotic normality.
In Section \ref{section-comparison},  we compare two \emph{centralized} DP estimates obtained by the Laplace mechanism with the \emph{local} DP estimation by the edge-flipping.
In Section \ref{section-numerical}, we carry out simulation studies and one real data analysis to evaluate
the theoretical results and compare the performance of the private estimates obtained by different types of perturbation mechanisms.
Section \ref{section-discussion} concludes the paper.
All proofs of the theorems are relegated to the online Supplementary Material.

\section{Preliminaries} \label{section-preliminaries}
\subsection{Centralized and Local Differential Privacy} \label{subsection-dp}

\emph{Differential privacy} (DP) was originally proposed for the centralized setting, so sometimes DP is also called \emph{centralized DP}.
\emph{Local DP}  (LDP) is designed in the local setting.
In this section, we introduce the differences and connections between these two types of DP.

The formal definition of $\epsilon$-\emph{centralized} DP is given as follows.
Consider an original database $D$ containing a set of records of $n$ individuals.
What we focus on is the mechanism that take $D$ as input and output a sanitized database $S$ for public use.
The size of $S$ may not be the same as $D$.
Denote $Q(\cdot)$ as a randomized mechanism for releasing data.
Let $\epsilon$  be a positive real number and $\mathcal{S}$ denote the sample space of $Q$.
The random mechanism $Q$ is \emph{ $\epsilon$-centralized differentially private} ($\epsilon$-\emph{centralized} DP)
if for any two neighboring databases $D_1$ and $D_2$ that differ on a single record,
and all measurable subsets $S$ of $\mathcal{S}$ [\cite{Dwork:Mcsherry:Nissim:Smith:2006}],
\[
\mathbb{P}(Q(D_1)\in S|D_1)\leq e^{\epsilon }\times \mathbb{P}(Q(D_2)\in S|D_2).
\]

The privacy parameter $\epsilon$, which is  publicly available, is chosen by the data curator
possessing the entire data.
It is used to control the trade-off
between privacy and utility  [\cite{Yan:2021}].
Note that a smaller value of $\epsilon$ means more privacy protection.

Centralized DP requires that the distribution of the output is almost the same, regardless of
whether an individual's record is in the dataset or not.
What is being protected in the centralized DP is precisely the difference between two neighboring databases. From a theoretical viewpoint, test statistics have nearly no power to test whether an individual's data are in the original database; see \cite{Wasserman:Zhou:2010} for a strict proof.

To break the centralized DP's limitation of relying on a trusted curator who has the whole original data, one extension is proposed as \emph{local DP}  (LDP), wherein each individual's data point undergoes perturbation with noise prior to data collection procedure. The formal definition of non-interactive $\epsilon$-LDP is provided as follows [\cite{Wang:2024}].

\begin{definition}[$\epsilon$-LDP]
\label{def-ldp}
 For a given privacy parameter $\epsilon>0$, the randomized mechanism $Q$  satisfies $\epsilon$-local differential privacy for an individual data point $X\in D$ if
\[
\sup_{\widetilde{x}\in \tilde{\mathcal{X}}}\sup_{x,x^\prime}
\frac{P(Q(X)=\widetilde{x}|X=x)}
{P(Q(X)=\widetilde{x}|X=x^\prime)}\leq e^{\epsilon}
\]
where $\tilde{\mathcal{X}}$ denotes the output space of $Q$.
\end{definition}
Note that in contrast to $\epsilon$-centralized DP, $\epsilon$-LDP does not depend on neighboring datasets and the trust data curator, which means that LDP privacy protection is stronger.

It's important to note that $\epsilon$-LDP can be analyzed under the framework of $\epsilon$-centralized DP in specific scenarios.
Consider the following setup:
 We define the output space as the product space $\mathcal{S}=\mathcal{X}^n$ (where $\mathcal{X}$ is the domain of a single data record and $n$ is the dataset size).
 Moreover, an $\epsilon$-LDP mechanism $Q$ is applied independently to each record within datasets consisting of independent samples.
 We have
$\mathbb P(Q(X)|X)=\prod_{i=1}^{n}\mathbb P(Q(X_i)|X_i)$, where $X=\{X_i\}_{i=1}^n \in \mathcal{X}^n$.
Under these conditions, for any neighboring datasets $D_1$ and $D_2$ differing only in the i-th record, we have
\[
\sup_{\tilde{D}\in \mathcal{X}^n}\frac{P(Q(D_1)=\tilde{D})}{P(Q(D_2)=\tilde{D})}=\sup_{\tilde{x}\in \mathcal{X}}\frac{P(Q(X_i)=\tilde{x}|X_i=x)}{P(Q(X_i)=\tilde{x}|X_i=x^\prime)}\leq e^{\epsilon},
\]
for $i\in [n]$.

LDP is mostly based on \emph{randomized responses}, which facilitates conducting questionnaires privately.
We introduce a real-life scenario to explain the randomized response mechanism $Q$ (known as the edge-flipping mechanism in networks):
We hope to calculate the proportion of AIDS and thus initiate a sensitive question: Are you an AIDS patient? Suppose there are $n$ users and each user needs to respond to this.
The real answer $X_i$ of the $i$-th user is yes or no (1 or 0), but for privacy reasons, the user will not directly respond to the real answer.
Suppose the answer is given with the help of an unfair coin, with probability $p$ of heads and probability $1-p$ of tails.
Flip the coin, and if it's heads, the user answers the true answer, otherwise, the user answers the opposite answer, that is
\begin{equation}\label{eq-edge filpping}
Q_{p}(X_i)=
\begin{cases}
   X_i~~,&\text{with probability p} ~;\\
   1-X_i~~,&\text{with probability 1-p} ~.
\end{cases}
\end{equation}
Clearly, it follows that
\[
\P(Q_{p}(X_i)=1)=p\P(X_i=1)+(1-p)(1-\P(X_i=1)).
\]

Motivated by the real-life problem that contains sensitive personal information, employing randomized response can release and analyze data in order to protect privacy while maintaining the validity of statistical results.

\subsection{Centralized differential privacy in network data}
Within network data,
depending on the definition of
the neighboring graph, \emph{(centralized) differential privacy} (centralized DP) is divided into \emph{node-differential privacy} (node-DP) [\cite{Kasiviswanathan:Nissim:Raskhodnikova:Smith;2013}] and \emph{edge-differential privacy} (edge-DP) [\cite{Nissim:Raskhodnikova:Smith:2007}].
Two graphs that differ in exactly one edge are called edge-neighboring graphs, while two node-neighboring graphs require that one graph can obtain another by removing one node and its adjacent edges.
Edge-DP protects edges from being detected, whereas node-DP protects not only nodes but also their
adjacent edges, which is a stronger privacy policy.
However, it may be infeasible to design algorithms that
both support node-DP and have good utility [\cite{Hay:2009}].
In this paper, we focus on edge-based definitions of DP, which aim to protect the privacy of relationships within the network.
In the following, we shall introduce edge-based DP for the centralized setting and the local setting respectively.

Let $\delta(G, G^\prime)$ be the number of edges
on which $G$ and $G^\prime$ differ.
The formal definition of edge-DP for the centralized setting is as follows.

\begin{definition}[$\epsilon$- edge DP]
Let $\epsilon>0$ be a privacy parameter. A randomized mechanism
$Q(\cdot)$ is $\epsilon$-edge (centralized) differentially private if
\[
\sup_{ G, G^\prime \in \mathcal{G}, \delta(G, G^\prime)=1 } \sup_{ S\in \mathcal{S}}  \frac{ P(Q(G)\in S) }{ P(Q(G^\prime)\in S) } \le e^\epsilon,
\]
where $\mathcal{G}$ is the set of all graphs of interest on $n$ nodes, and
$\mathcal{S}$ is the set of all possible outputs.
\end{definition}

Let $f:\mathcal{G}\rightarrow\R^k$ be a function. The global sensitivity [\cite{Dwork:Mcsherry:Nissim:Smith:2006}] of the function $f$, denoted $\Delta f$, is defined below.
\begin{definition}[Global Sensitivity]
\label{def-global sensitivity}
Let $f:\mathcal{G}\rightarrow\R^k$. The global sensitivity of $f$ is defined as
\[
\Delta(f)=\max_{\delta(G,G^\prime)=1}\|f(G)-f(G^\prime)\|_1
\]
where $\|\cdot\|_1$ is the $L_1$-norm.
\end{definition}
Global sensitivity measures the worst case difference between any two neighboring graphs. The magnitude of the noise added in the centralized DP algorithm depends crucially on the global sensitivity. If the outputs are the network statistics, then a simple and classical algorithm to guarantee edge-DP is the Laplace mechanism [\cite{Dwork:Mcsherry:Nissim:Smith:2006}], which adds the Laplace noise calibrated to the global sensitivity of $f$.

\begin{lemma}[Laplace Mechanism in \cite{Dwork:Mcsherry:Nissim:Smith:2006}]\label{lem-laplace}
 Suppose $f:\mathcal{G}\rightarrow\R^k$ is an output function in $\mathcal{G}$. Let $e_1,\dots,e_k$ be independent and identically distributed (i.i.d) Laplace random variables with density function $e^{-|x|/\lambda}/\lambda$. Then, the Laplace mechanism outputs $f(G)+(e_1,\dots,e_k)$ that are $\epsilon$-edge differentially private, where $\epsilon=-\Delta(f)\log{\lambda}$.
\end{lemma}
When $f(G)$ is integer, we can use a discrete Laplace random variable as the noise, as in \cite{Karwa:Slakovic:2016}, where it has the probability mass function:
\[
\P(X=x)=\frac{1-\lambda}{1+\lambda}\lambda^{|x|},\ x\in\{0,\pm 1,\dots\},\lambda \in (0,1).
\]
Lemma \ref{lem-laplace} still holds if the continuous Laplace distribution is replaced by the discrete Laplace distribution.
Note that the Laplace mechanism can release the summary statistic of interest, which is a kind of output perturbation mechanism.

One nice property of differential privacy is that any function of a differentially private mechanism is also differentially private [\cite{Dwork:Mcsherry:Nissim:Smith:2006}].
\begin{lemma}[\cite{Dwork:Mcsherry:Nissim:Smith:2006}]
\label{lemma:fg}
Let $f$ be an
output of an $\epsilon$-differentially private mechanism and $g$ be any function. Then
$g(f(G))$ is also $\epsilon$-differentially private.
\end{lemma}

By Lemma \ref{lemma:fg}, any post-processing done on an output of a differentially private mechanism is also differentially private.

\subsection{Local differential privacy in network data}
To our knowledge, LDP requires each user to independently perturb their own data locally prior to data collection. In a network, it is quite natural to regard a node as a user and a user's neighbor list, the set of nodes with which it shares an edge, is treated as its data.

Note that the popular definition of edge-LDP depends on the meaning of a node's neighbor list in \cite{Qin:2017}.
Traditional privacy definitions for networks have primarily focused on undirected graphs, where a node's neighbor list typically comprises only its incident edges. However, this simplified approach proves inadequate for directed networks, where each node exhibits both outgoing and incoming connection patterns that jointly define its structural role. Definition \ref{def-Pairwise Neighbors List} introduces the \emph{pairwise neighbor list}, which simultaneously considers both out-edges and in-edges for each node.
This comprehensive representation accurately captures the complete connectivity profile of nodes in directed graphs and establishes a more meaningful unit for privacy protection.

\begin{definition}[Pairwise Neighbor List]\label{def-Pairwise Neighbors List}
Let $A$ be the adjacency matrix of a directed graph $G_n$ with $n$ nodes $\{1,\dots,n\}$. The pairwise neighbor list for node $i$ ($i=1,\dots,n$) is $N_i=\{(a_{i,j},a_{j,i})|j>i;j=1,\dots,n\}$.
\end{definition}
To avoid the issue of a directed edge to which node it belongs, a naive idea is only to consider the out-edges or in-edges of each node, but in this way, the information of a node may not be fully utilized.
In definition \ref{def-Pairwise Neighbors List}, we have taken into account both the out-edges and in-edges of the directed graph simultaneously.
Undirected graphs still apply to the definition \ref{def-Pairwise Neighbors List}, that is, $N_i=\{a_{i,j}|j>i;j=1,\dots,n\}$, because undirected graphs can be regarded as a special case in which $a_{i,j}=a_{j,i}$.

Building upon this definition, we propose two complementary privacy definitions: Definition \ref{def-edge LDP} (pairwise edge LDP) and Definition \ref{def-pairwise edge ldp} (weak edge LDP).

\begin{definition}[$\epsilon$-pairwise edge LDP]
 \label{def-edge LDP}
A randomized mechanism $Q$ satisfies $\epsilon$-edge local differential privacy ($\epsilon$-edge LDP) if and only if for any two pairwise neighbor lists $N_i$ and $N_i^\prime$, such that $N_i$ and $N_i^\prime$ only differ in one pair, and any $\tilde{N}_i\in range(Q)$, we have:
\begin{equation}\label{eq-pairwise edge LDP}
 \frac{P(Q(N_i)=\tilde{N_i})}{P(Q(N_i^\prime)=\tilde{N_i})}  \leq e^{\epsilon} .
\end{equation}
\end{definition}
For undirected graphs, the equation \eqref{eq-pairwise edge LDP} is equivalent to
$\frac{P(Q(A_i)=\tilde{A_i})}{P(Q(A_i^\prime)=\tilde{A_i})}\leq e^\epsilon$, where $A_i$ is the $i$-th row of the adjacency matrix $A$.
Note that $\epsilon$-pairwise edge LDP ensures that if two vertices have pairwise neighbor lists that differ by one pair, they cannot be reliably distinguished based on the outputs from the randomized algorithm.

Under the framework of LDP, we consider a simpler version of pairwise edge-LDP in networks.

\begin{definition}[$\epsilon$-weak edge LDP]
\label{def-pairwise edge ldp}
Let $A$ denote the adjacency matrix of network with $n$ nodes. We say a randomized mechanism $Q$ satisfies $\epsilon$-weak edge local differential privacy if
\begin{equation}\label{eq-edge ldp}
 \sup_{\tilde{a}_{i,j}\in \tilde{\mathcal{X}}}
\sup_{a_{i,j},a_{i,j}^{\prime}}
   \frac{P\big(Q(a_{i,j},a_{j,i})=(\tilde{a}_{i,j},\tilde{a}_{j,i})|(a_{i,j},a_{j,i})\big)}{P\big(Q(a_{i,j}^\prime,a_{j,i}^\prime)=(\tilde{a}_{i,j},\tilde{a}_{j,i})|(a_{i,j}^\prime,a_{j,i}^\prime)\big)}\leq e^{\epsilon},
\end{equation}
for any $i,j\in [n]$, where $\tilde{\mathcal{X}}$ denotes the range of edges.
\end{definition}

This hierarchical definitional framework addresses the diverse requirements of practical applications.  The \emph{pairwise edge LDP} provides stronger privacy guarantees by requiring that entire neighbor lists remain indistinguishable after perturbation, while the \emph{weak edge LDP} offers protection at the individual edge-pair level with enhanced practical efficiency.

Intuitively, $\epsilon$-weak edge LDP and $\epsilon$-pairwise edge LDP are intrinsically connected to the $\epsilon$-edge DP. Particularly, if $Q$ satisfies $\epsilon$-weak edge LDP and is independently applied to $A$ entrywisely,
given the independence of $a_{i,j}$, we have
\[
\mathbb P(Q(A)|A)
=\prod_{i\in [n]}\P(Q(N_i)|N_i)
=\prod_{i<j}\mathbb P\big(Q(a_{i,j},a_{j,i})|(a_{i,j},a_{j,i})\big),
\]
which leads to
\begin{align}
\nonumber
\frac{P(Q(A)=\tilde{A}|A)}
{P(Q(A^\prime)=\tilde{A}|A^\prime)}
&=
\frac{P(Q(N_i)=\tilde{N}_i|N_i)}
{P(Q(N_i^\prime)=\tilde{N}_i|N_i^\prime)}
\\
\label{eq-relation 2 of DP}
&=
\frac{P\big(Q(a_{i,j},a_{j,i})=(\tilde{a}_{i,j},\tilde{a}_{j,i})|(a_{i,j},a_{j,i})\big)}{P\big(Q(a_{i,j}^\prime,a_{j,i}^\prime)=(\tilde{a}_{i,j},\tilde{a}_{j,i})|(a_{i,j}^\prime,a_{j,i}^\prime)\big)}\leq e^{\epsilon},
\end{align}
for $A$ and $A^\prime$ that differ only in the $(i,j)$-th element.
Equation \eqref{eq-relation 2 of DP} establishes that if a randomized mechanism $Q$ satisfies $\epsilon$-weak edge LDP and is applied independently to each edge pair, then the ratio of probabilities for any two neighboring adjacency matrices $A$ and $A^\prime$ is bounded by $e^\epsilon$. This directly implies that $Q$ also satisfies $\epsilon$-edge DP in the centralized setting. Therefore, the $\epsilon$-weak edge LDP condition is sufficient to ensure $\epsilon$-edge DP, bridging the local and centralized privacy frameworks. This result underscores that the edge-flipping mechanism, under $\epsilon$-weak edge LDP, not only protects individual edge privacy locally but also provides global privacy guarantees for the entire network release.
In addition, a similar correlation can be established between $\epsilon$-weak edge LDP and the $(k,\epsilon)$-edge DP, as explored in prior works such as \cite{Hay:2009} and \cite{Yan:2023}.

Consider $t$  pairwise (or weak) edge-LDP algorithms $Q_1,Q_2,\dots,Q_t$,
whose privacy budgets are denoted by $\epsilon_1,\epsilon_2,\dots,\epsilon_t$, respectively. We have the following property.
\begin{lemma}[Sequential Composition]
\label{prop-sequential composition}
The composite algorithm obtained by sequentially applying $Q_1,Q_2,\dots,Q_t$ on
the directed graph $G$ provides $\sum_{i=1}^t\epsilon_i$-pairwise (or weak) edge LDP.
\end{lemma}

\subsection{Generalized random response for directed networks}
For instance, in the edge-flipping algorithm \eqref{eq-edge filpping}, the probability of a user reporting their true answer is $p$, and the probability of reporting an opposite answer is $1-p$.

Random response is a fundamental technique in local differential privacy for collecting binary data (e.g., yes/no).  Its generalization, \emph{Generalized Randomized Response} (GRR), extends this to the multivariate case.

In directed networks, it is desirable to perturb the pair of edges between two nodes jointly, which can maintain the reciprocity structure.
A pair of edges between two nodes  takes four possible states: $(0,0),(0,1),(1,0),(1,1)$.
Let the true dyad be $R_{i,j}=(a_{i,j},a_{j,i})$ and the perturbed dyad be $T_{i,j}=(a_{i,j}^\prime,a_{j,i}^\prime)$.
Both $R_{i,j}$ and $T_{i,j}$ belong to the state space $\Omega = \{(0,0),(0,1),(1,0),(1,1)\}$.
GRR mechanisms are defined by using probability transition matrices.
The general form of the probability transition matrix $M$ is:
\[
M=
\begin{pmatrix}
 P_{00\rightarrow00} &P_{00\rightarrow01} &P_{00\rightarrow10} &P_{00\rightarrow11}\\
 P_{01\rightarrow00} &P_{01\rightarrow01} &P_{01\rightarrow10} &P_{01\rightarrow11}\\
 P_{10\rightarrow00} &P_{10\rightarrow01} &P_{10\rightarrow10} &P_{10\rightarrow11}\\
 P_{11\rightarrow00} &P_{11\rightarrow01} &P_{11\rightarrow10} &P_{11\rightarrow11}
\end{pmatrix}
\]
 where the sum of the elements in each row is $1$.
To meet the $\epsilon$-LDP requirements, for any two real dyads $R_{i,j}^{(1)}$ and $R_{i,j}^{(2)}$ as well as any perturbed dyad $T_{i,j}$, it is necessary to satisfy:
\begin{equation}
 e^{-\epsilon} \leq \frac{P(T_{i,j}|R_{i,j}^{(1)})}{P(T_{i,j}|R_{i,j}^{(2)})} \leq  e^{\epsilon}
\end{equation}
This means that for any two rows $i$ and $k$ of matrix $M$, as well as any column $j$:
\[e^{-\epsilon} \leq \frac{M_{i,j}}{M_{k,j}} \leq  e^{\epsilon}.\]

The simple GRR mechanism maintains the true state with a probability of $q$ and randomly selects one of the other three states uniformly with a probability of $1-q $ (each state has a probability of $(1-q)/3$). A general GRR mechanism is below.

\begin{definition}[Pairwise edge-flipping mechanism]
 \label{def-pairwise edge-flipping}
 Let the true dyad be $R_{ij}=(a_{i,j},a_{j,i})$ and the perturbed dyad be $T_{ij}=(a_{i,j}^\prime,a_{j,i}^\prime)$.
A pairwise edge-flipping mechanism $Q$ can be expressed as:
\begin{equation}\label{eq-GRR}
\P\big(T_{ij}=(a_{i,j}^\prime,a_{j,i}^\prime)|R_{ij}=(a_{i,j},a_{j,i})\big) =
\begin{cases}
 \gamma_1,\ if \ (a_{i,j}^\prime,a_{j,i}^\prime)=(a_{i,j},a_{j,i}) ,  \\
 \gamma_2,\ if \ a_{i,j}^\prime=a_{i,j} \ or \ a_{j,i}^\prime=a_{j,i}  ,   \\
 \gamma_3,\ if \  a_{i,j}^\prime\neq a_{i,j} \ and \ a_{j,i}^\prime \neq a_{j,i}.
\end{cases}
\end{equation}
where $a_{i,j},a_{j,i},a_{i,j}^\prime,a_{j,i}^\prime\in \{0,1\}$ and $\gamma_1+2\gamma_2+\gamma_3=1$.
\end{definition}
 Clearly, its probability transition matrix $M$ is
\[
M=
\begin{pmatrix}
  \gamma_1 &\gamma_2 & \gamma_2 & \gamma_3 \\
  \gamma_2 &\gamma_1 & \gamma_3 & \gamma_2 \\
  \gamma_2 &\gamma_3 & \gamma_1 & \gamma_2 \\
  \gamma_3 &\gamma_2 & \gamma_2 & \gamma_1,
\end{pmatrix}
\]
satisfying some restricted conditions.

\begin{lemma}
\label{thm-pairwise edge-ldp}
The pairwise edge-flipping mechanism $Q$ to satisfy $\epsilon$-weak edge LDP if
\begin{equation*}
e^{-\epsilon} \leq \frac{\gamma_1}{\gamma_2} \leq  e^{\epsilon},
e^{-\epsilon} \leq \frac{\gamma_1}{\gamma_3} \leq  e^{\epsilon} ,
e^{-\epsilon} \leq \frac{\gamma_2}{\gamma_3} \leq  e^{\epsilon}.
\end{equation*}
\end{lemma}

The proof of Lemma \ref{thm-pairwise edge-ldp} can be directly verified and therefore omitted. 
When 
\[
\P(T_{ij}=(c,d)|R_{ij}=(a,b))=\P(a_{i,j}^\prime=c|a_{i,j}=a)\cdot \P(a_{j,i}^\prime=d|a_{j,i}=b),
\]
pairwise edge-flipping mechanism reduces into perturbing all edges independently in a directed graph.
In this case, when the probability transition matrix $M$ is equal to
\[
M=\begin{pmatrix}
 p^2&p(1-p)&p(1-p)&(1-p)^2\\
 p(1-p)&p^2&(1-p)^2&p(1-p)\\
 p(1-p)&(1-p)^2&p^2&p(1-p)\\
 (1-p)^2&p(1-p)&p(1-p)&p^2
\end{pmatrix},
\]
the pairwise edge-flipping mechanism in \eqref{eq-GRR} is equivalent to the classical edge-flipping mechanism, where 
each edge $a_{i,j}$ is perturbed independently and the perturb probability is
\begin{equation*}
 \P(a_{i,j}^\prime|a_{i,j})=
 \begin{cases}
  p,\ &if \   a_{i,j}^\prime=a_{i,j},\\
  1-p,\  &if \  a_{i,j}^\prime=1-a_{i,j}.
 \end{cases}
\end{equation*}

\subsection{$p_0$ model}
\label{subsection-p0}

In this paper, we focus on  directed networks with binary edges, i.e. the edge weight takes value from $\{0,1\}$.
Let $G_n$ be a simple directed graph on $n\geq 2$ nodes that are labeled as ``1, \ldots, n."
Here, ``simple" means there are no multiple edges and no self-loops in $G_n$.
Let $A=(a_{i,j})$ be the adjacency matrix of $G_n$, where
$a_{i,j}$ is an indicator variable of the directed edge from head node $i$ to tail node $j$.
If there exists a directed edge from $i$ to $j$, then $a_{i,j}=1$; otherwise, $a_{i,j}=0$.
Because $G_n$ is loopless, we set $a_{i,i}=0$ for convenience.
Let $d_i^+= \sum_{j \neq i} a_{i,j}$ be the out-degree of node $i$
and $\boldsymbol{d}^+=(d_1^+, \ldots, d_n^+)^\top$ be the out-degree sequence of the graph $G_n$.
Similarly, define $d_i^- = \sum_{j \neq i} a_{j,i}$ as the in-degree of node $i$
and $\boldsymbol{d}^-=(d_1^-, \ldots, d_n^-)^\top$ as the in-degree sequence.
The pair $\boldsymbol{d}=( (\boldsymbol{d}^+)^\top, (\boldsymbol{d}^-)^\top)^\top$ or $\{(d_1^+, d_1^-), \ldots, (d_n^+, d_n^-)\}$ is called the bi-degree sequence.

The in- and out-degrees of vertices (or degrees for undirected networks) preliminarily summarize the information contained in a network, and their distributions provide important insights for understanding the generative mechanism of networks. As pointed by \cite{Hay:2009},  we may fail to protect privacy if we release the degree sequence directly, because some graphs have unique degree sequences.
In other scenarios, the bi-degrees of the nodes are themselves sensitive information.
For instance, the out-degree of an individual in a
sexually transmitted disease network reveals sensitive information,
such as how many people that person may have infected [\cite{Yan:2021}]. To protect edge privacy, we use the discrete Laplace mechanism (output perturbation) and the edge-flipping mechanism (input perturbation) to release the bi-degree sequence of the directed network.

To conduct statistical inferences from a noisy bi-degree sequence, we need to specify a model on the original bi-degree sequence.
We use the classical $p_0$ model to characterize the degree sequence.
The $p_0$ model can be represented as
\begin{equation}
\label{eq-p0model}
\P(G_n)= \frac{1}{c(\boldsymbol{\alpha}, \boldsymbol{\beta})} \exp( \sum_i \alpha_i d_i^+ + \sum_j \beta_j d_j^- ),
\end{equation}
where $c(\boldsymbol{\alpha}, \boldsymbol{\beta})$ is a normalizing constant, $\boldsymbol{\alpha}=(\alpha_1, \ldots, \alpha_n)^\top$,
and $\boldsymbol{\beta}=(\beta_1, \ldots, \beta_n)^\top$.
The outgoingness parameter $\alpha_i$ characterizes how attractive the node is, and the incomingness parameter $\beta_{i}$ illustrates the extent to which the node is attracted to others, as discussed in \cite{Holland:Leinhardt:1981}.
Although the $p_0$ model looks simple, it is still useful in applications
where only the bi-degree sequence is used. First, it serves as a null model for hypothesis testing [\cite{Holland:Leinhardt:1981,Fienberg:Wasserman:1981,Zhang:Chen:2013}]. Second, it can be
used to reconstruct networks and make statistical inferences when only the bi-degree sequence is available,
 owing to privacy considerations
[\cite{Helleringer:Kohler:2007,Shao:2021}]. Third, it can be used in a preliminary
analysis to choose suitable statistics for network configurations
[\cite{Robins.et.al.2009}].

Because an out-edge from node $i$ pointing to $j$ is the in-edge of $j$ coming from $i$, the sum of the out-degrees equals the sum of the in-degrees.
If one transforms $(\boldsymbol{\alpha}, \boldsymbol{\beta})$ to $(\boldsymbol{\alpha}+c, \boldsymbol{\beta}-c)$, the probability distribution in \eqref{eq-p0model} does not change.
To identify the model parameters, we set $\beta_n=0$, as in \cite{Yan:Leng:Zhu:2016}.
The $p_0$ model can be formulated using an array of mutually independent Bernoulli random variables $a_{i,j}$, $1\le i\neq j\le n$, with the following probabilities [\cite{Yan:Leng:Zhu:2016}]:
\[
\P(a_{i,j}=1) = \frac{ e^{\alpha_i + \beta_j} }{ 1 + e^{\alpha_i + \beta_j} }.
\]
The normalizing constant $c(\boldsymbol{\alpha}, \boldsymbol{\beta})$ in \eqref{eq-p0model} is $\sum_{i\neq j}\log( 1 + e^{\alpha_i+\beta_j} )$.

\section{Edge-flipping algorithm to release the network }\label{section-edge flipping}

The edge-flipping mechanism is a classical randomized response algorithm guaranteeing LDP at the edge level for networks [\cite{Karwa:2017}].
A great advantage of the edge-flipping mechanism is that it produces a synthetic network with clear distributional properties [\cite{Hehir:2022}].
Therefore, we use this simple yet effective randomized response mechanism to generate synthetic networks under $\epsilon$-weak edge LDP, and use its bi-degree sequence to estimate the parameter later.

Specifically, denote the flipped network as $Q_{p}(A)$ with a flipping probability $1-p$ for some $p\geq 1/2$.
Let $A^\prime=(a^\prime_{i,j})$ be the released adjacency matrix of $G_n$. By the edge-flipping mechanism in Lemma \ref{lem-edge flipping}, we have
\begin{equation}\label{eq-graph edge-flipping}
a^\prime_{i,j}=Q_{p}(a_{i,j})=
\begin{cases}
   a_{i,j},&\text{with probability p} ~,\\
   1-a_{i,j},&\text{with probability 1-p}.
\end{cases}
\end{equation}
The output $a^\prime_{i,j}$ is also a Bernoulli random variable with probability
\begin{equation}
\P(a^\prime_{i,j}=1)=\P(Q_{p}(a_{i,j})=1)=pP_{i,j}+(1-p)(1-P_{i,j}),
\end{equation}
where
$
P_{i,j}=\P(a_{i,j}=1),
1-P_{i,j}=\P(a_{i,j}=0).
$

As we know, when the edges of the network are independently perturbed, condition \eqref{eq-edge ldp} in weak edge-LDP is defined to be simplified to for any $i,j=1,\dots,n$,
\begin{equation}\label{eq-edge ldp independent}
 \sup_{\tilde{a}_{i,j}\in \tilde{\mathcal{X}}}
\sup_{a_{i,j},a_{i,j}^{\prime}}
   \frac{P\big(Q(a_{i,j})=\tilde{a}_{i,j}|a_{i,j}\big)}{P\big(Q(a_{i,j}^\prime)=\tilde{a}_{i,j}|a_{i,j}^\prime\big)}\leq e^{\epsilon}.
\end{equation}
If we set
\[
p=1/(1+e^{-\epsilon}),
\]
then the edge-flipping mechanism satisfies the $\epsilon$-weak edge LDP condition \eqref{eq-edge ldp independent}.

Unlike the Laplace mechanism in Lemma \ref{lem-laplace} which is one of the output perturbation methods, the edge-flipping mechanism is an input perturbation mechanism. This mechanism outputs the entire synthetic network denoted as an adjacency matrix, not just a summary statistic, which is certainly more attractive to us.

\begin{lemma}[Edge-flipping mechanism]\label{lem-edge flipping}
The edge-flipping mechanism $Q_{p}$ satisfies the $\epsilon$-weak edge-LDP when $p=(1+e^{-\epsilon})^{-1}$.
\end{lemma}
\begin{proof}
 For any $(i,j)\in [n]\times[n]$, we have
 \[
 \sup_{1\leq i<j\leq n}\sup_{\tilde{x}\in
 \{0,1\}}\sup_{x,\tilde{x}\in \{0,1\}}
 \frac{P(Q(a_{i,j})=\tilde{x}|a_{i,j}=x)}{P(Q(a_{i,j})=\tilde{x}|a_{i,j}=x^\prime)}=\max\Big\{1,\frac{p}{1-p},\frac{1-p}{p}\Big\}=e^{\epsilon},
 \]
 where the last inequality follows from the assumption that $p=(1+e^{-\epsilon})^{-1}$.
\end{proof}

Lemma \ref{lem-edge flipping} characterizes the capacity of the uniform random edge-flipping mechanism to protect privacy under the framework of $\epsilon$-weak edge LDP. It should be noted that privacy in $\mathcal{G}$ is completely protected when $p=1/2$ or $\epsilon=0$, in the sense that there exists no algorithm that can infer $a_{i,j}$ based on $Q_{1/2}(a_{i,j})$ better than random guessing [\cite{Wang:2024}].

The released steps are described in Algorithm \ref{algorithm:a}, which returns the entire flipped graph denoted as an adjacency matrix.
By Lemma \ref{lemma:fg}, the bi-degree sequence of the output graph is edge-LDP.

\begin{algorithm}
\caption{Releasing the adjacency matrix and the bi-degree sequence }
\label{algorithm:a}
\KwData{The adjacency matrix $A=(a_{i,j})$ of $G_n$ and privacy parameter $\epsilon_n$}
\KwResult{The flipped adjacency matrix $A^\prime=(a_{i,j}^\prime)$ and the local-differentially private bi-degree sequence $\boldsymbol{d}^\prime=(d_1^{'+},\ldots,d_{n}^{'+},d_1^{'-},\ldots,d_{n}^{'-})^\top$}
\For{ $i= 1 \to n$}{
 \For{$j= 1 \to n$}{
 Generates an $n\times n$ matrix $Y=(y_{i,j})$ whose elements are random numbers of 0 or 1 generated with probability p;

 If $y_{i,j}=1$, then $a_{i,j}^\prime=a_{i,j}$;

 If $y_{i,j}=0$, then $a_{i,j}^\prime=1-a_{i,j}$;

 $d_i^{'+}= \sum_{j \neq i} a_{i,j}^{'}$ \ ,\ $d_i^{'-}= \sum_{i \neq j} a_{i,j}^{'},\ (i=1,\cdots ,n)$.
 }
}
\end{algorithm}

We use the edge-flipping mechanism in Lemma \ref{lem-edge flipping} to
release the adjacency matrix 
and get the bi-degree sequence 
under $\epsilon$-weak edge LDP. Let $d_i^{'+}= \sum_{j \neq i} a_{i,j}^{'}$ be the flipped out-degree of node $i$, $d_j^{'-}= \sum_{i \neq j} a_{i,j}^{'}$ be the flipped in-degree of node $j$, and $\boldsymbol{d}^{'}=(d_1^{'},\ldots,d_{2n}^{'})^\top=(d_1^{'+},\ldots,d_{n}^{'+},d_1^{'-},\ldots,d_{n}^{'-})^\top$ be the flipped bi-degree sequence of the graph $G_n$. Then we obtain
\begin{equation}
\P(a_{i,j}^{'}=1)
=pP_{i,j}+(1-p)(1-P_{i,j})
=\frac{pe^{\alpha_i+\beta_j}+1-p}{1+e^{\alpha_i+\beta_j}}
\end{equation}
Later, we use the moment method to estimate the degree parameters in the $p_0$ model.

The mechanism releases a complete privacy-preserving network in the form of a perturbed adjacency matrix, despite its $O(n^2)$ complexity. Unlike mechanisms that only release aggregate statistics, our approach enables publication of the entire synthetic network.
Following appropriate debiasing procedures as described in \cite{Wang:2024},
this synthetic matrix preserves the expected structure of the original network while providing formal privacy guarantees.
It should be noted, however, that certain fine-grained structural patterns of the original network may not be fully recoverable due to the inherent randomness introduced for privacy protection [\cite{Wang:2024}]. Fundamentally, this methodology facilitates the release of richer network data, thereby supporting a broader spectrum of downstream analysis tasks beyond parameter estimation.

\section{Theoretical properties of the estimator }
\label{section-theoretical properties}

Having established the edge-flipping mechanism for releasing the entire network under $\epsilon$-weak edge LDP, a fundamental question arises: can valid statistical inference still be conducted from such a locally perturbed network? This section addresses this question by developing moment estimators for the parameters in the $p_0$ model using the noisy bi-degree sequence from the flipped graph. Given that the number of parameters grows with the network size $n$, the asymptotic analysis is inherently high-dimensional and non-standard. Our primary objective is to derive the theoretical guarantees for this LDP estimator, $\widehat{\boldsymbol{\theta}}_{in}$. We rigorously establish its consistency and asymptotic normality, demonstrating that despite the stronger privacy constraints and the substantial noise introduced at the input level, the estimator converges to the true parameter vector at a quantifiable rate. These results fill a critical theoretical gap by providing the first asymptotic guarantees for parameter estimation in network models under edge LDP.

We use the following moment equations to estimate the degree parameter in the $p_0$ model:
\renewcommand{\arraystretch}{1.5}
\begin{equation}\label{eq:likelihood-DP}
\begin{array}{lcl}
d_i^{\prime +}  & = & \sum_{j\neq i} \frac{pe^{\alpha_i+\beta_j}+(1-p)}{1+e^{\alpha_i+\beta_j}}, ~~i=1, \ldots, n, \\
d_j^{\prime - } & = & \sum_{i\neq j} \frac{pe^{\alpha_i+\beta_j}+(1-p)}{1+e^{\alpha_i+\beta_j}}, ~~j=1, \ldots, n-1,
\end{array}
\end{equation}
where $\boldsymbol{d}^\prime$ is the edge-LDP bi-degree sequence of Algorithm \ref{algorithm:a}.
The fixed point algorithm can be used to solve the above system of equations.
The solution $\widehat{\boldsymbol{\theta}}_{in}$ to the equations \eqref{eq:likelihood-DP} is the edge-LDP estimator of $\boldsymbol{\theta}$, according to Lemma \ref{lemma:fg},
where $\widehat{\boldsymbol{\theta}}_{in}=(\hat{\alpha}_1, \ldots, \hat{\alpha}_n, \hat{\beta}_1, \ldots, \hat{\beta}_{n-1} )^\top$
and $\hat{\beta}_n=0$.

We are curious whether the parameter estimate using the edge-LDP bi-degree sequence published by the edge-flipping mechanism still satisfies the consistency and asymptotic normality like those of the edge-DP estimators via output perturbation mechanisms in \cite{Yan:2021}.

Because the number of parameters increases with the number of nodes, classical statistical theories cannot be
applied directly to obtain the asymptotic results of the
estimator. We use the Newton method developed in \cite{Yan:Leng:Zhu:2016} to establish consistency. Here, we need to deal with the high-dimensional issue and the noise; in contrast, \cite{Yan:Leng:Zhu:2016} only considered the high-dimensional issue.
The proof for the existence and consistency of $\widehat{\boldsymbol{\theta}}_{in}$ can be briefly described as follows.
Define a system of functions:
\renewcommand{\arraystretch}{1.2}
\begin{equation}\label{eq:F-DP}
\large
\begin{array}{lll}
F_i( \boldsymbol{\theta} ) &  =  & \sum_{k=1; k \neq i}^n \frac{pe^{\alpha_i+\beta_k}+(1-p)}{1+e^{\alpha_i+\beta_k} }-d_i^{\prime+}, ~~~  i=1, \ldots, n, \\
F_{n+j}( \boldsymbol{\theta}) & = & \sum_{k=1; k\neq j}^n \frac{pe^{\alpha_k+\beta_j}+(1-p)}{1+e^{\alpha_k+\beta_j}}-d_j^{\prime-},  ~~~  j=1, \ldots, n, \\
F( \boldsymbol{\theta} ) & = & (F_1( \boldsymbol{\theta} ), \ldots, F_{2n-1}( \boldsymbol{\theta} ))^\top.
\end{array}
\end{equation}
Note that the solution to the equation $F( \boldsymbol{\theta} )=0$ is precisely the $\epsilon$-edge LDP estimator.
We construct the Newton iterative sequence: $\boldsymbol{\theta}^{(k+1)} = \boldsymbol{\theta}^{(k)} - [ F'(\boldsymbol{\theta}^{(k)})]^{-1} F(\boldsymbol{\theta}^{(k)})$. If the initial value is chosen as
the true value $\boldsymbol{\theta}^*$, then it is left to bound the error between the initial point
and the limiting point to show the consistency. This is done by establishing a geometric convergence rate for the iterative sequence.
The existence and consistency of $\widehat{\boldsymbol{\theta}}_{in}$ is stated below. In general, the privacy parameter $\epsilon_n$ is small. Therefore, we assume that $\epsilon_n$ is bounded by a fixed constant. This simplifies the notation.

\begin{theorem}\label{Thm-consistency}
Assume that $A \sim \P_{ \boldsymbol{\theta}^*}$, where $\P_{ \boldsymbol{\theta}^*}$ denotes
the probability distribution \eqref{eq-p0model} on $A$ under the parameter $\boldsymbol{\theta}^*$.
If $ \epsilon_n^{-1}e^{12\|\boldsymbol{\theta}^*\|_\infty } = o( (n/\log n)^{1/2} )$, 
 then with probability approaching one as $n$ goes to infinity, the estimator $\widehat{\boldsymbol{\theta}}_{in}$ exists
and satisfies
\[
\|\widehat{\boldsymbol{\theta}}_{in} - \boldsymbol{\theta}^* \|_\infty = O_p\left( \frac{ (\log n)^{1/2}e^{6\|\boldsymbol{\theta}^*\|_\infty} }{ n^{1/2}\epsilon_n } \right)=o_p(1).
\]
Furthermore, if $\widehat{\boldsymbol{\theta}}_{in}$ exists, it is unique.
\end{theorem}

\begin{remark}
First, the condition $\epsilon_n^{-1}e^{12\|\boldsymbol{\theta}^*\|_\infty } = o( (n/\log n)^{1/2} )$ in Theorem \ref{Thm-consistency} similarly
exhibits an interesting trade-off between the privacy parameter $\epsilon_n$ and $\|\boldsymbol{\theta}^*\|_\infty$ as \cite{Yan:2021}.
A smaller privacy parameter $\epsilon_n$ (stronger privacy) allows a larger $\|\boldsymbol{\theta}^*\|_\infty$
and if $\|\boldsymbol{\theta}^*\|_\infty$ is large, the sample size $n$ must increase or privacy constraints (via $\epsilon_n$) must relax to satisfy the condition.
Second, using the edge-flipping mechanism, we develop theoretical guarantees for edge-LDP estimation based on $p_0$ model, demonstrating conditions under which this privacy guarantee can be upheld while achieving convergence rates that match the known rates with edge-DP guarantee.
\end{remark}

In order to present the asymptotic normality, we introduce a
class of matrices.
Given two positive numbers $m$ and $M$ with $M \ge m >0$, we say the
$(2n-1)\times(2n-1)$ matrix $V=(v_{i,j})$ belongs to the class
$\mathcal{L}_{n}(m, M)$ if the following holds:
\begin{equation}\label{eq:LmM}
\begin{array}{l}
m\le v_{i,i}-\sum_{j=n+1}^{2n-1} v_{i,j} \le M, ~~ i=1,\ldots, n-1; ~~~ v_{n,n}=\sum_{j=n+1}^{2n-1} v_{n,j}, \\
v_{i,j}=0, ~~ i,j=1,\ldots,n,~ i\neq j, \\
v_{i,j}=0, ~~ i,j=n+1, \ldots, 2n-1,~ i\neq j,\\
m\le v_{i,j}=v_{j,i} \le M, ~~ i=1,\ldots, n,~ j=n+1,\ldots, 2n-1,~ j\neq n+i, \\
v_{i,n+i}=v_{n+i,i}=0,~~ i=1,\ldots,n-1,\\
v_{i,i}= \sum_{k=1}^n v_{k,i}=\sum_{k=1}^n v_{i,k}, ~~ i=n+1, \ldots, 2n-1.
\end{array}
\end{equation}

Clearly, if $V\in\mathcal{L}_{n}(m, M)$, then $V$ is a $(2n-1)\times
(2n-1)$ diagonally dominant, symmetric nonnegative
matrix and $V$ has the following structure:

\[
V=\begin{pmatrix}
  V_{11} & V_{12}
\\
V_{12}^\top&V_{22}
\end{pmatrix}
\]
where $V_{11}$ ($n$ by $n$) and $V_{22}$ ($n-1$ by $n-1$) are diagonal
matrices, $V_{12}$ is a nonnegative matrix whose nondiagonal elements
are positive and diagonal elements equal to zero.

Define $v_{2n,i}=v_{i,2n}:= v_{i,i}-\sum_{j=1;j\neq i}^{2n-1} v_{i,j}$
for $i=1,\ldots, 2n-1$ and $v_{2n,2n}=\sum_{i=1}^{2n-1} v_{2n,i}$.
Then $m \le v_{2n,i} \le M$ for $i=1,\ldots, n-1$, $v_{2n,i}=0$ for
$i=n, n+1,\ldots, 2n-1$ and $v_{2n,2n}=\sum_{i=1}^n v_{i, 2n}=\sum
_{i=1}^n v_{2n, i}$. We propose to approximate the inverse of $V$,
$V^{-1}$, by the matrix $S=(s_{i,j})$, which is defined as
\begin{equation}\label{eq-S}
s_{i,j}=\left\{\begin{array}{ll}\frac{\delta_{i,j}}{v_{i,i}} + \frac{1}{v_{2n,2n}}, & i,j=1,\ldots,n, \\
-\frac{1}{v_{2n,2n}}, & i=1,\ldots, n,~~ j=n+1,\ldots,2n-1, \\
-\frac{1}{v_{2n,2n}}, & i=n+1,\ldots,2n-1,~~ j=1,\ldots,n, \\
\frac{\delta_{i,j}}{v_{i,i}}+\frac{1}{v_{2n,2n}}, & i,j=n+1,\ldots, 2n-1,
\end{array}
\right.
\end{equation}
where $\delta_{i,j}=1$ when $i=j$, and $\delta_{i,j}=0$ when $i\neq j$.
Note that $S$ can be rewritten as
\[
S=\begin{pmatrix}
  S_{11} & S_{12}
\\
S_{12}^\top&S_{22}
\end{pmatrix}
\]
where $S_{11} =1/v_{2n, 2n} + \operatorname{diag}(1/v_{1,1}, 1/v_{2,2}, \ldots,
1/v_{n,n})$, $S_{12}$ is an $n\times(n-1)$ matrix whose elements are
all equal to $-1/v_{2n, 2n}$, and $S_{22} = 1/v_{2n, 2n} + \operatorname{diag}(1/v_{n+1, n+1},  \ldots, 1/v_{2n-1, 2n-1})$.

We use $V$ to denote the Jacobian matrix of $F(\boldsymbol{\theta})$ based on $p_0$ model.
It can be shown that
\[
v_{i,j}=\frac{ (2p-1)e^{\alpha_i + \beta_j} }{ (1 + e^{\alpha_i + \beta_j})^2 },~~ 1\le i\neq j \le n;
\]
\begin{equation}\label{eq-vii}
v_{i,i}=
\begin{cases}
   \sum_{k\neq i}\frac{(2p-1)e^{\alpha_i+\beta_k}}{(1+e^{\alpha_i+\beta_k})^2},\ i=1,\cdots,n\\
   \sum_{k\neq i-n}\frac{(2p-1)e^{\alpha_k+\beta_{i-n}}}{(1+e^{\alpha_k+\beta_{i-n}})^2},\ i=n+1,\cdots,2n

\end{cases}
\end{equation}
Because $e^x/(1+e^x)^2$ is an increasing function on $x$ when $x\ge 0$, and
a decreasing function when $x\le 0$, we have
\[
\frac{(2p-1)(n-1)e^{2\| \theta\|_\infty}}{(1+e^{2\| \theta\|_\infty})^2}\le v_{i,i} \le \frac{(2p-1)(n-1)}{4}, ~~ i=1, \ldots, 2n.
\]
Therefore, $V\in \mathcal{L}_n(m,M)$, where $m$ is the left expression and $M$ is the right expression in the above inequality.
The asymptotic distribution of $\widehat{\boldsymbol{\theta}}_{in}$ depends on $V$.
Let $\boldsymbol{g}=(d_1^+, \ldots, d_n^+, d_1^-, \ldots, d_{n-1}^-)^\top$ and
$\boldsymbol{g}^\prime=(d_1^{\prime+}, \ldots, d_n^{\prime+}, d_1^{\prime-}, \ldots, d_{n-1}^{\prime-})^\top$.
If we apply Taylor's expansion to each component of $\boldsymbol{g}^\prime -\E \boldsymbol{g}^\prime$, then the second-order term in the expansion is $V(\widehat{\boldsymbol{\theta}}_{in} - \boldsymbol{\theta})$.
Because $V^{-1}$ does not have a closed form, we work with $S$ defined at \eqref{eq-S} to approximate it. Then, we represent $\widehat{\boldsymbol{\theta}}_{in} - \boldsymbol{\theta}$ as the sum of
$S(\boldsymbol{g}^\prime -\E \boldsymbol{g}^\prime)$ and a remainder. The central limit theorem is proved by establishing the asymptotic normality of $S(\boldsymbol{g}^\prime -\E \boldsymbol{g}^\prime)$ and
showing that the remainder is negligible. Let $\boldsymbol{\sigma}^2=(\sigma_1^2,\ldots,\sigma_n^2,\sigma_{n+1}^2,\ldots,\sigma_{2n}^2)$, for any $i=1,\ldots, 2n$, we have
\[
\begin{aligned}
\mathrm{Var}(d_i^{'}) =\sigma_i^2 &=
\begin{cases}
\mathrm{Var}(d_i^{'+}) =\mathrm{Var}(\sum_{k\neq i}a_{i,k}^{'})
 =\sum_{k\neq i}\mathrm{Var}(a_{i,k}^{'})
 \\
\mathrm{Var}(d_{i-n}^{'-}) =\mathrm{Var}(\sum_{k\neq i-n}a_{k,i-n}^{'})
 =\sum_{k\neq i-n}\mathrm{Var}(a_{k,i-n}^{'})
\end{cases}\\
&=
\begin{cases}
\sum_{k\neq i}\frac{[pe^{\alpha_i+\beta_k}+1-p][(1-p)e^{\alpha_i+\beta_k}+p]}{(1+e^{\alpha_i+\beta_k})^2}~~,~~1\leq i \leq n  \\
\sum_{k\neq i-n}\frac{[pe^{\alpha_k+\beta_{i-n}}+1-p][(1-p)e^{\alpha_k+\beta_{i-n}}+p]}{(1+e^{\alpha_k+\beta_{i-n}})^2}
~~,~~n+1\leq i \leq 2n
\end{cases}
\end{aligned}
\]
Clearly, it shows that $[pe^{\alpha_i+\beta_k}+1-p][(1-p)e^{\alpha_i+\beta_k}+p]\geq e^{\alpha_i+\beta_k}$, that is
$\mathrm{Var}(a_{i,k}^\prime)\geq\mathrm{Var}(a_{i,k})$, then we have  $\sigma_i^2>v_{i,i}$ and $\frac{\sigma_i^2}{v_{i,i}^2}+\frac{\sigma_{2n}^2}{v_{2n,2n}^2}>\frac{1}{v_{i,i}}+\frac{1}{v_{2n,2n}}$. This indicates that the variance of the estimator obtained by adding noise is larger than that of the estimator obtained without adding noise.

We formally state the asymptotic normality of $\widehat{\boldsymbol{\theta}}_{in}$ as follows.
\begin{theorem}\label{Thm-normality}
Assume that $A\sim \P_{\boldsymbol{\theta}^*}$, where $\P_{\boldsymbol{\theta}^*}$ denotes the probability distribution \eqref{eq-p0model} on $A$ under the true parameter $\boldsymbol{\theta}^*$ .\\
If $\epsilon_n^{-1}e^{18\|\boldsymbol{\theta}^*\|_\infty}
=o(n^{1/2}(\log{n})^{-1})$,
then for any fixed $k\ge 1$, as $n \to\infty$, the vector consisting of the first $k$ elements of $(\widehat{\boldsymbol{\theta}}_{in}-\boldsymbol{\theta}^*)$ is asymptotically multivariate normal with mean $\mathbf{0}$ and covariance matrix given by the upper left $k \times k$ block of the matrix
\(M=\begin{pmatrix}
  M_{11} & M_{12}\\
  M_{12}^\top&M_{22}\end{pmatrix}\)
  , where
 $M_{11} =\sigma_{2n}^2/v_{2n, 2n}^2 + \operatorname{diag}(\sigma_1^2/v_{1,1}^2, \ldots,
\sigma_n^2/v_{n,n}^2)$, $M_{12}$ is an $n\times(n-1)$ matrix whose elements are all equal to
$-\sigma_{2n}^2/v_{2n, 2n}^2$, and $M_{22} = \sigma_{2n}^2/v_{2n, 2n}^2 + \operatorname{diag}(\sigma_{n+1}^2/v_{n+1, n+1}^2, \ldots, \sigma_{2n-1}^2/v_{2n-1, 2n-1}^2)$. \\
\end{theorem}

\begin{remark}
 If we change the first $k$ elements of $(\widehat{\boldsymbol{\theta}}_{in}-\boldsymbol{\theta}^*)$ to an arbitrarily fixed $k$ elements with the subscript set $\{i_1, \ldots, i_k \}$,
Theorem \ref{Thm-normality} still holds. This is because all steps in the proof are valid if we change the first $k$ subscript set from $\{1, \ldots, k\}$ to $\{i_1, \ldots, i_k \}$.
\end{remark}

\section{Comparison to edge differential privacy}
\label{section-comparison}

The theoretical results in Section \ref{section-theoretical properties} confirm that asymptotically valid inference is feasible under the local privacy model. This naturally leads to a subsequent, pragmatic inquiry: how does the local DP approach compare with the more established centralized DP paradigm in terms of data utility and estimation efficiency?
While these frameworks are designed for different trust scenarios, a systematic comparison is crucial for understanding the practical cost of the enhanced privacy offered by LDP and the value of the richer data it releases.

In this section, we juxtapose our edge-LDP estimator ($\widehat{\boldsymbol{\theta}}_{in}$) with its centralized edge-DP counterparts--specifically, the estimators obtained via the Laplace mechanism with ($\widehat{\boldsymbol{\theta}}_{de-Lap}$) and without ($\widehat{\boldsymbol{\theta}}_{Lap}$) denoising.
The goal is not to declare a universal winner, but to quantitatively characterize the trade-off between privacy strength, estimation accuracy, and information richness.
We demonstrate that the ability to release an entire synthetic network under LDP, which supports diverse downstream tasks, comes at the cost of a larger asymptotic variance compared to the centralized DP mechanism that releases only a denoised summary statistic. This comparison provides a clear framework for practitioners to choose the appropriate privacy technology based on their specific requirements for privacy, utility, and data functionality.

We first introduce the nonprivate estimator in \cite{Yan:Leng:Zhu:2016} as the baseline  and then show two edge-DP estimators in the centralized setting through the output mechanisms in \cite{Yan:2021}.

Regardless of privacy, let $\boldsymbol{\theta}=(\alpha_1, \ldots, \alpha_n, \beta_1, \ldots, \beta_{n-1})^\top$ and $\beta_n=0$, the likelihood equations are
\begin{equation}\label{eq-mle}
\begin{array}{lcl}
d_i^+  & = & \sum_{j\neq i} \frac{e^{\hat{\alpha}_i+\hat{\beta}_j}}{1+e^{\hat{\alpha}_i+\hat{\beta}_j}}, ~~i=1, \ldots, n, \\
d_j^-  & = & \sum_{i\neq j} \frac{e^{\hat{\alpha}_i+\hat{\beta}_j}}{1+e^{\hat{\alpha}_i+\hat{\beta}_j}}, ~~j=1, \ldots, n-1,
\end{array}
\end{equation}
where $\widehat{\boldsymbol{\theta}}_{mle}=(\hat{\alpha}_{mle_1},\dots,\hat{\alpha}_{mle_n},\hat{\beta}_{mle_1},\dots,\hat{\beta}_{mle_{n-1}})^\top$ is the MLE of $\boldsymbol{\theta}$ and $\hat{\beta}_{mle_n}=0$.
Note that the solution $\widehat{\boldsymbol{\theta}}_{mle}$ to the estimate equation is fitted in the $p_0$ model with the original bi-degree sequence $\boldsymbol{d}$.

In order to protect the information of the edges in the network, as proposed by \cite{Yan:2021}, we use the output perturbation mechanism such as the discrete Laplace mechanism in Lemma \ref{lem-laplace} to release the bi-degree sequence and conduct inferences using the noisy bi-sequence.

Let $\boldsymbol{z}=(\boldsymbol{z}^+,\boldsymbol{z}^-)$ be the released bi-sequence via the discrete Laplace mechanism, then we have $z_i^+=d_i^++e_i^+,z_i^-=d_i^-+e_i^-$, where $e_i^+,e_i^-$ independently follow discrete Laplace with $\lambda=e^{-\epsilon/2}$. Next, we replace the original bi-degree sequence $(\boldsymbol{d}^+,\boldsymbol{d}^-)$ in \eqref{eq-mle} with the noisy bi-sequence $(\boldsymbol{z}^+,\boldsymbol{z}^-)$ and use the following equations
to estimate the  parameter:
\begin{equation}\label{eq-nondenoise}
\begin{array}{lcl}
z_i^+  & = & \sum_{j\neq i} \frac{e^{\alpha_i+\beta_j}}{1+e^{\alpha_i+\beta_j}}, ~~i=1, \ldots, n, \\
z_j^-  & = & \sum_{i\neq j} \frac{e^{\alpha_i+\beta_j}}{1+e^{\alpha_i+\beta_j}}, ~~j=1, \ldots, n-1,
\end{array}
\end{equation}
where $\boldsymbol{z}=(\boldsymbol{z}^+,\boldsymbol{z}^-)$ is the differentially private bi-sequence. Because the discrete Laplace distribution is symmetric with mean zero, the above equations are also the moment equations. The solution $\widehat{\boldsymbol{\theta}}_{Lap}$ to the equations \eqref{eq-nondenoise} is the edge differentially private estimator of $\boldsymbol{\theta}$, according to Lemma \ref{lemma:fg}, where $\widehat{\boldsymbol{\theta}}_{Lap}=(\hat{\alpha}_{Lap_1},\dots,\hat{\alpha}_{Lap_n},\hat{\beta}_{Lap_1},\dots,\hat{\beta}_{Lap_{n-1}})$ and $\hat{\beta}_{Lap_n}=0$.

However, the output $\boldsymbol{z}$ through the Laplace mechanism is generally not the graphical bi-degree sequence. To make $\boldsymbol{z}$ graphical, we need to denoise $\boldsymbol{z}$ like \cite{Yan:2021}. Let $B_n$ be the set of all possible bi-degree sequences of graph $G_n$. It is natural to use the closest point $\widehat{\boldsymbol{d}}_{de} $ lying in $B_n$ as the denoised bi-sequence, with some distance between $\widehat{\boldsymbol{d}}_{de} $ and $\boldsymbol{d}$. We use the $L_1$-distance here, and define the estimator as
\[
\widehat{\boldsymbol{d}}_{de}=arg\min_{\boldsymbol{d}\in B_n}(\|\boldsymbol{z}^+-\boldsymbol{d}^+\|_1+\|\boldsymbol{z}^--\boldsymbol{d}^-\|_1).
\]

Note that both $\boldsymbol{z}=\widehat{\boldsymbol{d}}_{Lap}$ and $\widehat{\boldsymbol{d}}_{de}$ are edge-DP estimators of $\boldsymbol{d}$. We can replace $\boldsymbol{z}$ with $\widehat{\boldsymbol{d}}_{de}$ in the equations in \eqref{eq-nondenoise} to obtain the denoised estimator of the parameter $\boldsymbol{\theta}$; denote the solution as $\widehat{\boldsymbol{\theta}}_{de-Lap}$. By repeatedly using Lemma \ref{lemma:fg}, $\widehat{\boldsymbol{\theta}}_{Lap}$ and $\widehat{\boldsymbol{\theta}}_{de-Lap}$ are both edge-DP estimators.

The consistency and asymptotic normality for the maximum likelihood estimator $\widehat{\boldsymbol{\theta}}_{mle}$ have been established in \cite{Yan:Leng:Zhu:2016} and those of the two edge-DP estimators $\widehat{\boldsymbol{\theta}}_{Lap},\widehat{\boldsymbol{\theta}}_{de-Lap}$ have been established in \cite{Yan:2021}.

In the following, we compare the difference between the nonprivate estimator $\widehat{\boldsymbol{\theta}}_{mle}$, two edge-DP estimators $\widehat{\boldsymbol{\theta}}_{Lap}$, $\widehat{\boldsymbol{\theta}}_{de-Lap}$ and the edge-LDP estimator $\widehat{\boldsymbol{\theta}}_{in}$ based on $p_0$ model.
 Recall that $\widehat{\boldsymbol{\theta}}_{mle}$ is nonprivate estimate using the original bi-degree sequence $\boldsymbol{d}$, $\widehat{\boldsymbol{\theta}}_{Lap}$ and $\widehat{\boldsymbol{\theta}}_{de-Lap}$ are $\epsilon$-edge DP estimates via the output perturbation, and $\widehat{\boldsymbol{\theta}}_{in}$ is $\epsilon$-edge LDP estimate via input perturbation.
$\widehat{\boldsymbol{\theta}}_{Lap}$ is estimated using the bi-sequence $\boldsymbol{z}$ published by the discrete Laplace mechanism, while $\widehat{\boldsymbol{\theta}}_{de-Lap}$ is estimated using the bi-degree sequence $\widehat{\boldsymbol{d}}_{de}$ denoised by $\boldsymbol{z}$.
Note that the above four estimates $\widehat{\boldsymbol{\theta}}_{mle},\widehat{\boldsymbol{\theta}}_{Lap},\widehat{\boldsymbol{\theta}}_{de-Lap},\widehat{\boldsymbol{\theta}}_{in}$ all have consistency and asymptotic normality.

For $\widehat{\boldsymbol{\theta}}_{in}$, we use $V$ to denote the Jacobian matrix of $F(\boldsymbol{\theta})$ in \eqref{eq:F-DP}.
For clarity, in view of \eqref{eq-mle} and \eqref{eq-nondenoise}, for $\widehat{\boldsymbol{\theta}}_{mle},\widehat{\boldsymbol{\theta}}_{Lap}$ and $\widehat{\boldsymbol{\theta}}_{de-Lap}$, we represent the corresponding Jacobian matrix with $\widetilde{V}$,
where
\[
v_{i,j}=\frac{ (2p-1)e^{\alpha_i + \beta_j} }{ (1 + e^{\alpha_i + \beta_j})^2 }=(2p-1)\tilde{v}_{i,j} ~, 1\le i\neq j \le n \ ;\ v_{i,i}=(2p-1)\tilde{v}_{i,i} ~,1\leq i\leq 2n.
\]

We compare how these methods differ in terms of published data and statistical inference in Table \ref{table-all theta.hat}.
We denote $s_n^2=Var(\sum_{i=1}^{n}e_i^{+}-\sum_{i=1}^{n-1}e_i^{-})=(2n-1)\frac{2e^{-\epsilon_n/2}}{(1-e^{-\epsilon_n/2})^2}$, where $e_i^+,e_i^-$ independently follow discrete Laplace with $\lambda=e^{-\epsilon/2}$  and let $R_{p}=O_p\left(\frac{(\log n)^{1/2}e^{6\|\boldsymbol{\theta}^*\|_{\infty}}}{n^{1/2}}\right)$.

\begin{table}[h]\centering
\caption{Comparison of estimators under different private mechanisms in the $p_0$ model}
\label{table-all theta.hat}
\scriptsize
\begin{tabular}{ccccc}
\hline
& non-private & output perturbation & output perturbation & input perturbation \\
& MLE        & non-denoised Laplace         &  denoised Laplace        & edge-flipping      \\
\hline
\rule{0pt}{12pt} released data & $\boldsymbol{d}$&
$\widehat{\boldsymbol{d}}_{Lap}=\boldsymbol{z}$&
$\widehat{\boldsymbol{d}}_{de}$&
$A^\prime,\widehat{\boldsymbol{d}}_{in}=\boldsymbol{d}^\prime$ \\
\rule{0pt}{12pt} notation & $\widehat{\boldsymbol{\theta}}_{mle}$ &$\widehat{\boldsymbol{\theta}}_{Lap}$
& $\widehat{\boldsymbol{\theta}}_{de-Lap}$ & $\widehat{\boldsymbol{\theta}}_{in}$ \\

\rule{0pt}{15pt} convergence rate
& $e^{2\|\boldsymbol{\theta}^*\|_{\infty}}\cdot R_p$
& $\frac{1}{\epsilon_n}\cdot R_p$
& $R_p$                & $\frac{1}{\epsilon_n}\cdot R_p$                   \\
\rule{0pt}{15pt} asymptotic variance
&  $\frac{1}{\tilde{v}_{i,i}}+\frac{1}{\tilde{v}_{2n,2n}}$
&  $\frac{1}{\tilde{v}_{i,i}}+\frac{1}{\tilde{v}_{2n,2n}}+\frac{s_n^2}{\tilde{v}_{2n,2n}^2}$
&  $\frac{1}{\tilde{v}_{i,i}}+\frac{1}{\tilde{v}_{2n,2n}}$
&  $\frac{\sigma_i^2}{{v}_{i,i}^2}+\frac{\sigma_{2n}^2}{{v}_{2n,2n}^2}$                  \\
\hline
\end{tabular}
\end{table}
From Table \ref{table-all theta.hat}, we can see that in terms of releasing data, under edge-LDP, the input perturbation mechanism called the edge-flipping mechanism can publish more information, which releases a complete synthetic network, not just a sequence of degrees.
Note that $\frac{\sigma_i^2}{{v}_{i,i}^2}+\frac{\sigma_{2n}^2}{{v}_{2n,2n}^2}>\frac{1}{\tilde{v}_{i,i}}+\frac{1}{\tilde{v}_{2n,2n}}$, which means that the variance of $\widehat{\boldsymbol{\theta}}_{in}$ is bigger than that of $\widehat{\boldsymbol{\theta}}_{de-Lap}$.
From Figure \ref{fig-var-comp},
except for $(\alpha_1,\beta_1)$, the variance of $\widehat{\boldsymbol{\theta}}_{in}$ is larger than that of $\widehat{\boldsymbol{\theta}}_{de-Lap}$ but smaller than that of $\widehat{\boldsymbol{\theta}}_{Lap}$, which is consistent with the theoretical conclusion. The large variance of $(\alpha_1,\beta_1)$ might be due to the large $\|\boldsymbol{\theta}^*\|_{\infty}$ in this case.
To some extent, compared with the output perturbation mechanism for releasing degree sequence, the edge-flipping mechanism adds relatively greater noise, which means that accuracy is somewhat sacrificed.
However, the estimation using the edge-flipping bi-degree sequence still maintains consistency and asymptotic normality when the error is acceptable.

\begin{figure}[!h]
\centering
\includegraphics[ height=3in, width=4in, angle=0]
{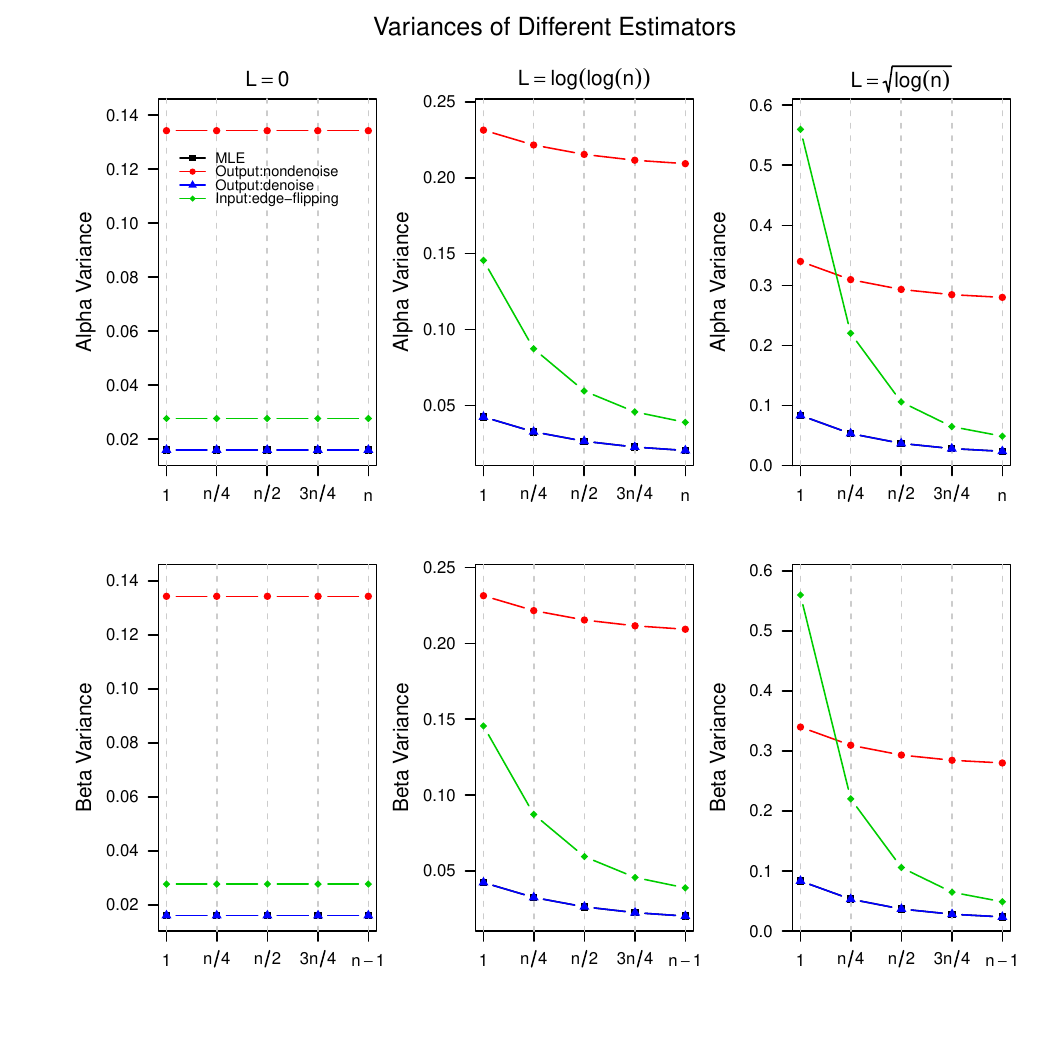}

\caption{Variance comparison of the four estimators in the $p_0$ model }
\label{fig-var-comp}
\end{figure}

Based on the comprehensive comparison presented in Table \ref{table-all theta.hat} and Figure \ref{fig-var-comp}, we can draw the following conclusive insights regarding the trade-offs between different private mechanisms in the $p_0$ model estimation:

The choice between centralized and local differential privacy frameworks presents a fundamental trade-off between statistical efficiency and data utility. While the denoised Laplace mechanism under centralized DP achieves estimation accuracy comparable to the non-private MLE by releasing a denoised bi-degree sequence, the edge-flipping mechanism under local DP, despite its larger asymptotic variance, enables the release of an entire synthetic network. This capacity to support diverse downstream analytical tasks beyond parameter estimation justifies the marginal loss in estimation precision, positioning local DP not merely as a solution for decentralized scenarios where centralized DP is inapplicable, but as a valuable alternative when richer data publication is required, even at the cost of slightly degraded statistical efficiency.

\section{Numerical studies}
\label{section-numerical}
In this part, we conducted simulation studies and one real data analysis to assess and compare the performance of the estimators obtained by private output and input perturbation mechanisms.
Note that $\widehat{\boldsymbol{\theta}}_{Lap}$ and $\widehat{\boldsymbol{\theta}}_{de-Lap}$ are both edge-DP estimators,
where $\widehat{\boldsymbol{\theta}}_{Lap}$ denotes the edge-DP estimator corresponding to the noisy bi-sequence $\widehat{\boldsymbol{d}}_{Lap}=\boldsymbol{z}$ released by the discrete Laplace mechanism, and $\widehat{\boldsymbol{\theta}}_{de-Lap}$ denotes the edge-DP estimator corresponding to the denoised bi-degree sequence $\widehat{\boldsymbol{d}}_{de}$.
Moreover, $\widehat{\boldsymbol{\theta}}_{in}$ denotes the edge-LDP estimator corresponding to the noisy bi-degree sequence $\widehat{\boldsymbol{d}}_{in}=\boldsymbol{d}^\prime$ released by the edge-flipping mechanism.
We assess the performance of the estimators in finite sizes of networks when
$n$, $\epsilon_n$ or the range of $\theta_i$ vary
and compare the numerical results of the edge-LDP estimator $\widehat{\boldsymbol{\theta}}_{in}$ by input perturbation mechanism with those of $\widehat{\boldsymbol{\theta}}_{Lap}$ and $\widehat{\boldsymbol{\theta}}_{de-Lap}$ by output perturbation mechanisms.

\subsection{Simulation studies}
\label{subsection-simulation}
In this section,
we evaluate the asymptotic results for model \eqref{eq:F-DP} by using the edge-flipping mechanism in Algorithm \ref{algorithm:a} and compare different private parameter estimators of $p_0$ model through numerical simulations.

The parameters in the simulations are as follows.
Similar to \cite{Yan:Leng:Zhu:2016}, the setting of the parameter $\boldsymbol{\theta}^*$ takes a linear form.
Specifically, we set $\alpha_{i+1}^* = (n-1-i)L/(n-1)$ for $i=0, \ldots, n-1$.
For the parameter values of $\boldsymbol{\beta}$, let $\beta_i^*=\alpha_i^*$, $i=1, \ldots, n-1$ for simplicity and $\beta_n^*=0$ by default.
We considered four different values for $L$, $L=0$, $\log(\log n)$, $(\log n)^{1/2}$ and $\log n$, respectively.
We simulated three different values for $\epsilon_n$: one is fixed ($\epsilon_n=2$) and the other two values tend to zero with $n$, i.e., $\epsilon_n=\log (n)/n^{1/4}, \log(n)/n^{1/2}$.
We considered three values for $n$, $n=100, 200$ and $500$.
Each simulation was repeated $10,00$ times.

By Theorem \ref{Thm-normality}, $\hat{\xi}_{i,j} = [\hat{\alpha}_i-\hat{\alpha}_j-(\alpha_i^*-\alpha_j^*)]/(\hat{\sigma}_i^2/\hat{v}_{i,i}^2+\hat{\sigma}_j^2/\hat{v}_{j,j}^2)^{1/2}$, $\hat{\zeta}_{i,j} = (\hat{\alpha}_i+\hat{\beta}_j-\alpha_i^*-\beta_j^*)/(\hat{\sigma}_i^2/\hat{v}_{i,i}^2+\hat{\sigma}_{n+j}^2/\hat{v}_{n+j,n+j}^2)^{1/2}$, and $\hat{\eta}_{i,j} = [\hat{\beta}_i-\hat{\beta}_j-(\beta_i^*-\beta_j^*)]/(\hat{\sigma}_{n+i}^2/\hat{v}_{n+i,n+i}^2+\hat{\sigma}_{n+j}^2/\hat{v}_{n+j,n+j}^2)^{1/2}$
converge in distribution to the standard normal distributions, where $\hat{v}_{i,i}$ is the estimate of $v_{i,i}$ by replacing $\boldsymbol{\theta}^*$ with $\widehat{\boldsymbol{\theta}}_{in}$.  Therefore, we assess the asymptotic normality of $\hat{\xi}_{i,j}$, $\hat{\zeta}_{i,j}$ and $\hat{\eta}_{i,j}$ using the quantile-quantile (QQ) plot.
The results for $\hat{\xi}_{i,j}$, $\hat{\zeta}_{i,j}$ and $\hat{\eta}_{i,j}$ are similar, thus only the results of $\hat{\xi}_{i,j}$ are reported.
By \cite{Yan:2021},
$\bar{\xi}_{i,j} = [\hat{\alpha}_{Lap,i}-\hat{\alpha}_{Lap,j}-(\alpha_{i}^*-\alpha_{j}^*)]/(1/\hat{\tilde{v}}_{i,i}^2+1/\hat{\tilde{v}}_{j,j}^2)^{1/2}$
and $\bar{\bar{\xi}}_{i,j} = [\hat{\alpha}_{de,i}-\hat{\alpha}_{de,j}-(\alpha_{i}^*-\alpha_{j}^*)]/(1/\hat{\tilde{v}}_{i,i}^2+1/\hat{\tilde{v}}_{j,j}^2)^{1/2}$also converge in distribution to the standard normal distribution, where $\hat{\tilde{v}}_{i,i}$ is the plug-in estimate of $\tilde{v}_{i,i}$.
We also draw the QQ-plots for 
$\bar{\xi}_{i,j}$ and $\bar{\bar{\xi}}_{i,j}$. Further, we compare $\hat{\xi}_{i,j}$ with 
$\bar{\xi}_{i,j}$ and $\bar{\bar{\xi}}_{i,j}$ to assess the performance of the estimators obtained by different perturbation mechanisms.
The distances between the original bi-degree sequence $\boldsymbol{d}$ and the noisy bi-sequence $\widehat{\boldsymbol{d}}_{Lap}$, the denoised bi-degree sequence $\widehat{\boldsymbol{d}}_{de}$, and the flipped bi-degree sequence $\widehat{\boldsymbol{d}}_{in}$ are also reported and represented as $\|\boldsymbol{d} - \widehat{\boldsymbol{d}}_{Lap}\|_\infty$, $\|\boldsymbol{d} - \widehat{\boldsymbol{d}}_{de}\|_\infty$ and $\|\boldsymbol{d} - \widehat{\boldsymbol{d}}_{in}\|_\infty$ respectively.

The average values of the $\ell_\infty$-distance between the original sequence and the released sequences are reported in Table \ref{tab-distance}.
We can see that the distance becomes larger as $\epsilon_n$ decreases. It means that smaller $\epsilon_n$ provides more privacy protection.
For example, when $\epsilon_n$ changes from $\log n/n^{1/4}$ to $\log n/n^{1/2}$, $\|\boldsymbol{d} - \widehat{\boldsymbol{d}}_{Lap}\|_\infty$ dramatically increases from $8$ to $26$ in the case $n=100$.
As expected,  the distance also becomes larger as $n$ increases when $\epsilon_n$ is fixed. As is understood, stronger privacy protection and more information released for LDP mean more noise is added, and thus its error will be relatively larger.

\begin{table}[!h]
\centering
\scriptsize
\caption{$\ell_\infty$-distances between original and released bi-degree sequences
under different privacy budgets $\epsilon_n$}

\label{tab-distance}
\begin{tabular}{cccccccccccc}
\hline
 & \multicolumn{11}{c}{$\epsilon_n$} \\
\cline{2-12}
 & \multicolumn{3}{c}{$2$} & &\multicolumn{3}{c}{$\log n/n^{1/4}$} &
&\multicolumn{3}{c}{$\log n/n^{1/2}$}  \\
\cline{2-4} \cline{6-8} \cline{10-12} 
\addlinespace
 $n$&$\hat{\boldsymbol{d}}_{Lap}$
 &$\hat{\boldsymbol{d}}_{de}$
 &$\hat{\boldsymbol{d}}_{in}$ &
&$\hat{\boldsymbol{d}}_{Lap}$
&$\hat{\boldsymbol{d}}_{de}$
&$\hat{\boldsymbol{d}}_{in}$ &
&$\hat{\boldsymbol{d}}_{Lap}$
&$\hat{\boldsymbol{d}}_{de}$
&$\hat{\boldsymbol{d}}_{in}$\\
$100$ & 5.7&11.0 &16.9   && 8.0&15.0&23.2 &  & 25.5&36.5&39.2 \\
$200$ & 6.4&14.8&29.7  && 9.2&21.2&43.1 &
& 35.1&63.1&76.3 \\
$500$ & 7.4&21.1&64.3  && 11.3&32.4&102.5 &
& 53.8&129.0&186.4 \\
\hline
\end{tabular}
\end{table}

In the QQ-plots, the horizontal and vertical axes are the theoretical and empirical quantiles, respectively,
and the straight lines correspond to the reference line $y=x$.
In all the QQ-plots, we show $\hat{\xi}_{i,j}$ with edge-flipping (in green color), $\bar{\xi}_{i,j}$ without denoised process (in black color) and $\bar{\bar{\xi}}_{i,j}$ with denoised process (in red color).
Because the results for QQ-plots are similar for $\epsilon_n=2,\log n/n^{1/4},\log n/n^{1/2};n=100,200,500$, only the QQ-plots for $\epsilon_n=2$ when $n=100,500$ are shown here to save space. The other plots are shown in the Supplementary Material.

When $\epsilon_n=2$, the QQ-plots for $\hat{\xi}_{i,j}$, $\bar{\xi}_{i,j}$ and $\bar{\bar{\xi}}_{i,j}$ under $n=100,500$ are shown in Figure \ref{fig-n1-ep1} and Figure \ref{fig-n3-ep1}.
By comparing the QQ-plots when $\epsilon_n=2$, we first observe that the
empirical quantiles of $\hat{\xi}_{i,j}$, $\bar{\xi}_{i,j}$ and $\bar{\bar{\xi}}_{i,j}$ all agree well with the theoretical ones when $L=0$.
Second, for $\hat{\xi}_{i,j}$, we observe that there are obvious
deviations for pair $(1,2)$ and moderate deviations for pair $(n/2,n/2+1)$ when $L=(\log n)^{1/2}$, while the deviations of $\bar{\bar{\xi}}_{i,j}$ are notable for pair $(n-1,n)$.
We also find the performance of $\hat{\xi}_{i,j}$ is much better than that
of $\bar{\bar{\xi}}_{i,j}$ for pair $(n-1,n)$, while the performance of $\bar{\bar{\xi}}_{i,j}$ is much better than that of $\hat{\xi}_{i,j}$ for pair $(1,2)$.
Third, it is shown that $\bar{\xi}_{i,j}$ performs much better than that of $\bar{\bar{\xi}}_{i,j}$ and $\hat{\xi}_{i,j}$, whose QQ-plots deviate from the diagonal line at both ends.
Understandably, because the edge-flipping mechanism under edge-LDP releases more information about the network, unlike the Laplace mechanism which only publishes the degree sequence, the edge-flipping mechanism adds relatively more noise, which leads to a relatively larger deviation of $\hat{\xi}_{ij}$.
Moreover, the deviation of the QQ-plots from the straight line becomes smaller as $n$ increases.
These phenomena corroborate with the findings in \cite{Yan:2021}.

When $\epsilon_n=\log n/n^{1/4},\log n/n^{1/2}$, the QQ-plots are displayed in the supplementary Material.
We observe that under $\epsilon_n=\log n/n^{1/2}$ and $L=(\log n)^{1/2}$, the estimate without denoise and the estimate with denoise did not exist in all repetitions, thus only the corresponding QQ-plot of the estimate via the edge-flipping mechanism could be shown.
In this case, the condition in Theorem \ref{Thm-normality} fails and these figures show significant deviations from the standard normal distribution.
It indicates that $\epsilon_n$ should not go to zero quickly as $n$ increases in order to guarantee good utility.
Lastly, we observe that when $L=\log n$  for which the condition in Theorem 2 of \cite{Yan:2021} fails, the nondenoised and denoised estimates did not exist in all repetitions.
Thus the corresponding QQ-plot could not be shown.

\begin{figure}[!htbp]
\centering
\includegraphics[ height=3in, width=4in, angle=0]{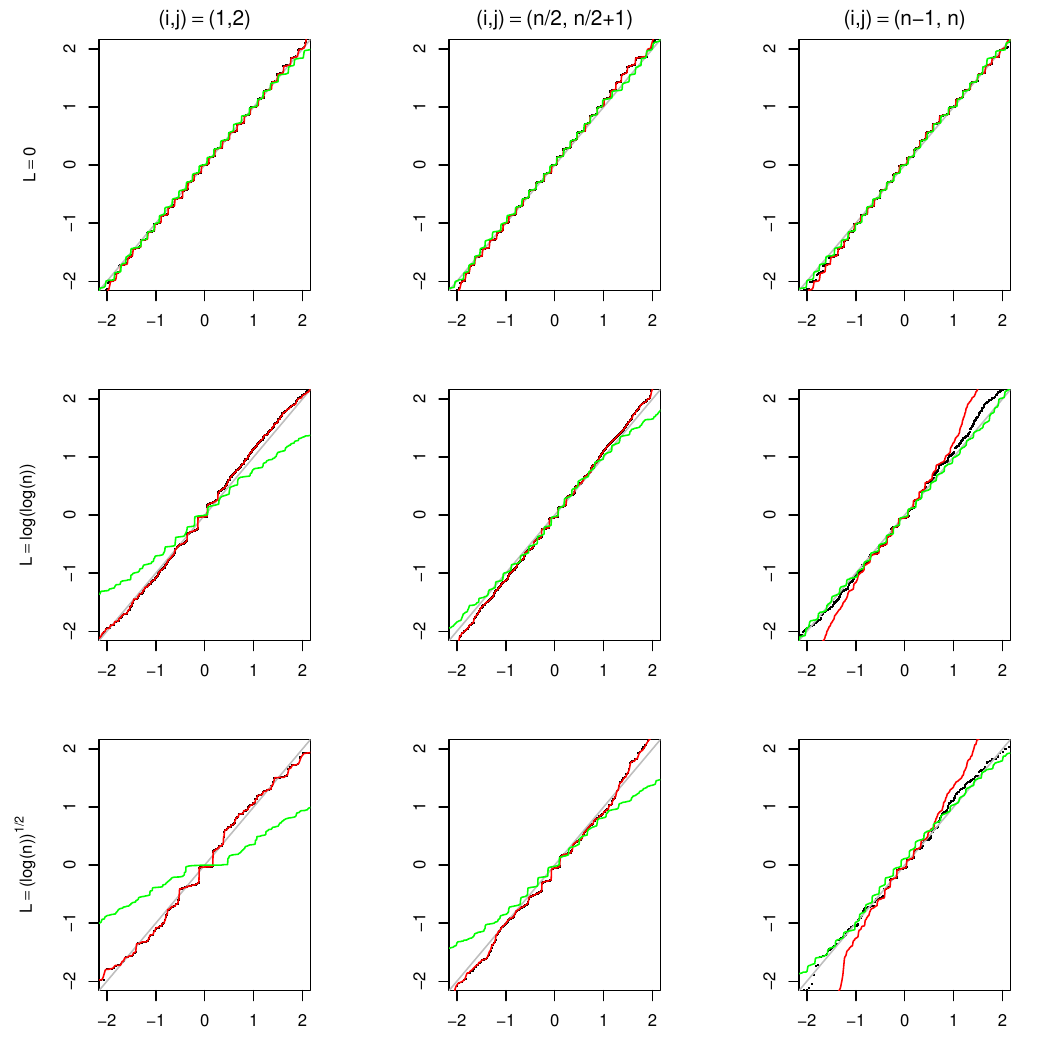}
\captionsetup{justification=centering}
\caption{QQ plots for $\xi_{i,j}$ under $\epsilon_n=2$ and $n=100$, comparing edge-flipping (green), non-denoised Laplace (black), and denoised Laplace (red) mechanisms}
\label{fig-n1-ep1}
\end{figure}

\begin{figure}[!htbp]
\centering
\captionsetup{justification=centering}
\includegraphics[ height=3in, width=4in, angle=0]{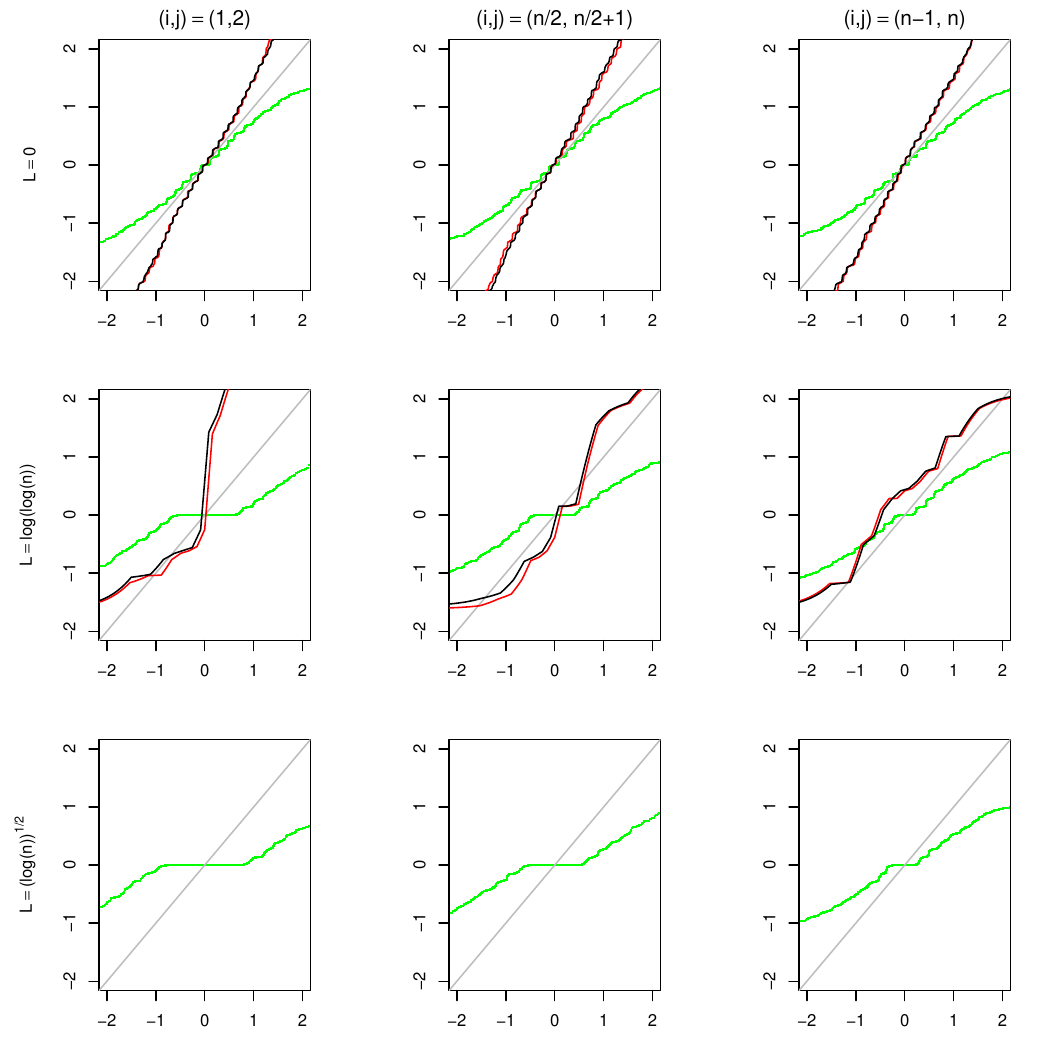}
\caption{QQ plots for $\xi_{i,j}$ under $\epsilon_n=2$ and $n=500$, comparing edge-flipping (green), non-denoised Laplace (black), and denoised Laplace (red) mechanisms}
\label{fig-n3-ep1}
\end{figure}

\newpage
\subsection{Real-data analysis}
\label{subsection-real data}

We evaluate how close the estimators $\boldsymbol{\hat{\theta}}_{Lap}=(\boldsymbol{\hat{\alpha}}_{Lap}, \boldsymbol{\hat{\beta}}_{Lap})$, $\boldsymbol{\hat{\theta}}_{de}=(\boldsymbol{\hat{\alpha}}_{de}, \boldsymbol{\hat{\beta}}_{de})$ and $\boldsymbol{\hat{\theta}}_{in}=(\boldsymbol{\hat{\alpha}}_{in}, \boldsymbol{\hat{\beta}}_{in})$ are to the MLE $\boldsymbol{\hat{\theta}}_{mle}=(\boldsymbol{\hat{\alpha}}_{mle}, \boldsymbol{\hat{\beta}}_{mle})$, fitted in the $p_0$ model with the original bi-degree sequence using one real network dataset: the UC Irvine messages data.
We present the analytical results of the UC Irvine messages data.
Recall that $(\boldsymbol{\hat{\alpha}}_{Lap}, \boldsymbol{\hat{\beta}}_{Lap})$ use the bi-sequence $\boldsymbol{\hat{d}}_{Lap}$ released by the discrete Laplace mechanism , $(\boldsymbol{\hat{\alpha}}_{de}, \boldsymbol{\hat{\beta}}_{de})$ use the denoised bi-sequence  $\boldsymbol{\hat{d}}_{de}$ and $(\boldsymbol{\hat{\alpha}}_{in}, \boldsymbol{\hat{\beta}}_{in})$ correspond to the bi-degree sequence $\boldsymbol{\hat{d}}_{in}=\boldsymbol{d}^\prime$ released by the edge-flipping mechanism.
Note that $(\boldsymbol{\hat{\alpha}}_{Lap}, \boldsymbol{\hat{\beta}}_{Lap})$ and $(\boldsymbol{\hat{\alpha}}_{de}, \boldsymbol{\hat{\beta}}_{de})$ by the output perturbation mechanisms are the edge-DP estimators of the vector parameters $\boldsymbol{\alpha}$ and $\boldsymbol{\beta}$, while $(\boldsymbol{\hat{\alpha}}_{in}, \boldsymbol{\hat{\beta}}_{in})$ by input perturbation mechanism is the edge-LDP estimator.
We set $\epsilon_n$ equal to one, two and three, as in \cite{Karwa:Slakovic:2016}, and 
repeat $1,000$ times for each $\epsilon_n$.
Then, we computed the average and the variance of $\|\boldsymbol{\hat{\alpha}}-\boldsymbol{\hat{\alpha}}_{mle}\|_{\infty},\|\boldsymbol{\hat{\beta}}-\boldsymbol{\hat{\beta}}_{mle}\|_{\infty}$ for three estimators $(\boldsymbol{\hat{\alpha}}_{Lap}, \boldsymbol{\hat{\beta}}_{Lap})$, $(\boldsymbol{\hat{\alpha}}_{de}, \boldsymbol{\hat{\beta}}_{de})$ and $(\boldsymbol{\hat{\alpha}}_{in}, \boldsymbol{\hat{\beta}}_{in})$.

The UC Irvine messages network data were created from an online community of students at the University of California, Irvine [\cite{Opsahl:2009}]. This network dataset covers the period from April to October 2004.
It has a total of $1,899$ nodes, and each node represents a student. During the observation period, students sent a total number of $59,835$ online messages.
A directed edge is established from one
student to another if one or more messages have been sent from
the former to the latter [\cite{Panzarasa:2009}].
In total, there are $20,296$ edges, and the edge density is $0.56\%$, indicating a very sparse network, and an average degree of $10.69$.
Of the $1,899$ nodes, $586$ nodes have no out-edges or in-edges.
Following \cite{Yan:2021},
we remove these nodes, because a nonprivate MLE does not exist in this case.
To guarantee nonzero out-degrees and in-degrees after adding noise with a large probability,
we analyze a subgraph with out-degrees and in-degrees both larger than five.
After data preprocessing, only $696$ nodes remain.
The quantiles of $0$, $1/4$, $1/2$, $3/4$ and $1$ are  $3$, $8$, $14$, $26$, and $164$ for the out-degrees, and $4$, $10$, $16$, $27$ and $121$ for the in-degrees, respectively.

When many nodes have few links to others,  large noise easily causes output with nonpositive elements in the discrete Laplace mechanism.
When $\epsilon=1$, the average $\ell_\infty$-distance between $\boldsymbol{d}$ and $\boldsymbol{\hat{d}}_{Lap}$ is $15.6$, and all private estimates fail to exist as in \cite{Yan:2021}. Therefore, we set $\epsilon=\log n/n^{1/4}$ ($\approx 1.27$), $\epsilon=2$ and $\epsilon=3$.
The results are shown in Table \ref{tab-real data}.
From this table, we can see that the mean values 
of $\|\boldsymbol{\hat{\alpha}}-\boldsymbol{\hat{\alpha}}_{mle}\|_{\infty}$ and $\|\boldsymbol{\hat{\beta}}-\boldsymbol{\hat{\beta}}_{mle}\|_{\infty}$, which indicate the estimators are very close to the MLE.
In addition, compared with the two edge-DP estimations $\boldsymbol{\hat{\theta}}_{Lap},\boldsymbol{\hat{\theta}}_{de}$, the error of the edge-LDP estimation $\boldsymbol{\hat{\theta}}_{in}$ is slightly larger because the edge-LDP mechanism releases more information, not just the bi-degree sequence of the graph. We also find that the frequencies for which the edge-LDP estimate fails to exist are significantly lower than those for the two edge-DP estimates. This is probably because the LDP edge-flipping mechanism directly releases a new graph, which ensures that the released bi-degree sequence must be graphical. However, a drawback of output perturbations for releasing is easy to generate zero or negative degrees, which results in a released sequence that may not be graphical.

\begin{table}[htbp]
\centering
\caption{Performance of private estimators on the UC Irvine message network: the average $\ell_\infty$ distance from the MLE with failure frequencies ($\times 100\%$) in parentheses}
\scriptsize
\begin{tabular}{cllccc}
\hline
Perturbation Type & Mechanism & Estimators & $\epsilon=\log n/n^{1/4}$ & $\epsilon=2$ & $\epsilon=3$ \\
\hline
\multirow{4}{*}{Output} & \multirow{2}{*}{non-denoised Laplace} & $\hat{\alpha}_{Lap}$ & 1.94 (99.4) & 1.42 (55.0) & 0.91 (9.6) \\
 & & $\hat{\beta}_{Lap}$ & 1.68 (99.4) & 1.28 (55.0) & 0.80 (9.6) \\
 \cline{2-6}
 & \multirow{2}{*}{denoised Laplace} & $\hat{\alpha}_{de-Lap}$ & 2.24 (99.5) & 1.62 (78.7) & 1.09 (54.7) \\
 & & $\hat{\beta}_{de-Lap}$ & 1.40 (99.5) & 1.22 (78.7) & 0.79 (54.7) \\
\hline
\multirow{2}{*}{Input} & \multirow{2}{*}{edge-flipping} & $\hat{\alpha}_{in}$ & 4.85 (0.0) & 4.64 (0.0) & 5.92 (0.0) \\
 & & $\hat{\beta}_{in}$ & 5.56 (0.0) & 4.91 (0.0) & 5.62 (0.0) \\
\hline
\end{tabular}
\label{tab-real data}
\end{table}

\section{Discussion}
\label{section-discussion}

Under edge-based LDP, we develop the edge-flipping mechanism to release directed networks
and use the released edge-flipping bi-degree sequence to estimate the parameter of $p_0$ model.
We establish the consistency and asymptotic normality of the edge LDP estimator.
Moreover, we compare different estimators obtained by the classical output mechanisms under edge-DP and the input mechanism under edge-LDP in terms of data releasing and statistical inference.
Compared with the output perturbation mechanism for releasing degree sequence, the edge-flipping mechanism adds relatively greater noise, which means that accuracy is somewhat sacrificed.
However, the estimation using the edge-flipping bi-degree sequence still maintains consistency and asymptotic normality when the error is acceptable.

The conditions in Theorems \ref{Thm-consistency} and \ref{Thm-normality} induce an interesting trade-off between the private parameter
and the growing rate of the parameter $\boldsymbol{\theta}$.
If the parameter $\epsilon_n$ is large, $\boldsymbol{\theta}$ can be allowed to be relatively large.
Moreover, the condition in Theorem \ref{Thm-normality} is much stronger than that in Theorem \ref{Thm-consistency}.
The asymptotic behavior of the estimator is not only determined by the growing rate of the parameter $\boldsymbol{\theta}$, but also by the configuration of the parameter. It would be of interest to see whether these conditions could be relaxed.

We only make inference from the noisy data in a simple $p_0$ model.
In order to extend the method of deriving the consistency of the estimator in our paper to other network models,
one needs to establish a geometrical rate of convergence of the Newton iterative sequence.
This is not easy for network models with other network features since it is difficult to derive the upper bound of the matrix norm for the inverse matrix of the Fisher information matrix without some special matrix structures.
At the same time, it is also difficult to extend the method of deriving asymptotic normality of the estimator to network models with other network features
since it is generally difficult to derive the approximate inverse matrix of a general Fisher information matrix.
Therefore, it is of interest to develop new methods and techniques to analyze LDP estimators in other network models.

\setlength{\itemsep}{-1.5pt}
\setlength{\bibsep}{0ex}


\begin{thebibliography}{}

\bibitem[Backstrom et~al., 2011]{Backstrom:Dwork:Kleinberg:2011}
Backstrom, L., Dwork, C., and Kleinberg, J. (2011).
\newblock Wherefore art thou r3579x? anonymized social networks, hidden patterns, and structural steganography.
\newblock {\em Communications of the ACM}, 54(12):133--141.

\bibitem[Bebensee, 2019]{Bebensee:2019}
Bebensee, B. (2019).
\newblock Local differential privacy: a tutorial.
\newblock {\em arXiv preprint arXiv:1907.11908}.

\bibitem[Chang et~al., 2024]{Chang:2024}
Chang, J., Hu, Q., Kolaczyk, E.~D., Yao, Q., and Yi, F. (2024).
\newblock Edge differentially private estimation in the $\beta$-model via jittering and method of moments.
\newblock {\em The Annals of Statistics}, 52(2):708--728.

\bibitem[Duchi et~al., 2013]{Duchi:2013}
Duchi, J.~C., Jordan, M.~I., and Wainwright, M.~J. (2013).
\newblock Local privacy and statistical minimax rates.
\newblock In {\em 2013 IEEE 54th Annual Symposium on Foundations of Computer Science}, pages 429--438. IEEE.

\bibitem[Dwork et~al., 2006]{Dwork:Mcsherry:Nissim:Smith:2006}
Dwork, C., McSherry, F., Nissim, K., and Smith, A. (2006).
\newblock Calibrating noise to sensitivity in private data analysis.
\newblock In {\em Theory of Cryptography Conference}, pages 265--284. Springer.

\bibitem[Erlingsson et~al., 2014]{Erlingsson:2014}
Erlingsson, {\'U}., Pihur, V., and Korolova, A. (2014).
\newblock Rappor: Randomized aggregatable privacy-preserving ordinal response.
\newblock In {\em Proceedings of the 2014 ACM SIGSAC Conference on Computer and Communications Security}, pages 1054--1067.

\bibitem[Fienberg and Wasserman, 1981]{Fienberg:Wasserman:1981}
Fienberg, S.~E. and Wasserman, S. (1981).
\newblock An exponential family of probability distributions for directed graphs: Comment.
\newblock {\em Journal of the American Statistical Association}, 76(373):54--57.

\bibitem[Hay et~al., 2009]{Hay:2009}
Hay, M., Li, C., Miklau, G., and Jensen, D. (2009).
\newblock Accurate estimation of the degree distribution of private networks.
\newblock In {\em 2009 Ninth IEEE International Conference on Data Mining}, pages 169--178. IEEE.

\bibitem[Hehir et~al., 2022]{Hehir:2022}
Hehir, J., Slavkovi{\'c}, A., and Niu, X. (2022).
\newblock Consistent spectral clustering of network block models under local differential privacy.
\newblock {\em The Journal of Privacy and Confidentiality}, 12(2):1--22.

\bibitem[Helleringer and Kohler, 2007]{Helleringer:Kohler:2007}
Helleringer, S. and Kohler, H.-P. (2007).
\newblock Sexual network structure and the spread of hiv in africa: evidence from likoma island, malawi.
\newblock {\em Aids}, 21(17):2323--2332.

\bibitem[Holland and Leinhardt, 1981]{Holland:Leinhardt:1981}
Holland, P.~W. and Leinhardt, S. (1981).
\newblock An exponential family of probability distributions for directed graphs.
\newblock {\em Journal of the American Statistical Association}, 76(373):33--50.

\bibitem[Imola et~al., 2021]{Imola:2021}
Imola, J., Murakami, T., and Chaudhuri, K. (2021).
\newblock Locally differentially private analysis of graph statistics.
\newblock In {\em 30th USENIX Security Symposium (USENIX Security 21)}, pages 983--1000.

\bibitem[Karwa et~al., 2017]{Karwa:2017}
Karwa, V., Krivitsky, P.~N., and Slavkovi{\'c}, A.~B. (2017).
\newblock Sharing social network data: differentially private estimation of exponential family random-graph models.
\newblock {\em Journal of the Royal Statistical Society Series C: Applied Statistics}, 66(3):481--500.

\bibitem[Karwa and Slavkovi\'c, 2016]{Karwa:Slakovic:2016}
Karwa, V. and Slavkovi\'c, A. (2016).
\newblock Inference using noisy degrees: Differentially private $\beta$-model and synthetic graphs.
\newblock {\em The Annals of Statistics}, 44(1):87--112.

\bibitem[Kasiviswanathan et~al., 2013]{Kasiviswanathan:Nissim:Raskhodnikova:Smith;2013}
Kasiviswanathan, S.~P., Nissim, K., Raskhodnikova, S., and Smith, A. (2013).
\newblock Analyzing graphs with node differential privacy.
\newblock In {\em Theory of Cryptography Conference}, pages 457--476. Springer.

\bibitem[Lu and Miklau, 2014]{Lu:Miklau:2014}
Lu, W. and Miklau, G. (2014).
\newblock Exponential random graph estimation under differential privacy.
\newblock In {\em Proceedings of the 20th ACM SIGKDD International Conference on Knowledge Discovery and Data Mining}, pages 921--930.

\bibitem[Narayanan and Shmatikov, 2009]{Narayanan:Shmatikov:2009}
Narayanan, A. and Shmatikov, V. (2009).
\newblock De-anonymizing social networks.
\newblock In {\em 2009 30th IEEE Symposium on Security and Privacy}, pages 173--187. IEEE.

\bibitem[Nguyen et~al., 2016]{Nguyen:Imine:Rusinowitch:2016}
Nguyen, H.~H., Imine, A., and Rusinowitch, M. (2016).
\newblock Detecting communities under differential privacy.
\newblock In {\em Proceedings of the 2016 ACM on Workshop on Privacy in the Electronic Society}, pages 83--93.

\bibitem[Nissim et~al., 2007]{Nissim:Raskhodnikova:Smith:2007}
Nissim, K., Raskhodnikova, S., and Smith, A. (2007).
\newblock Smooth sensitivity and sampling in private data analysis.
\newblock In {\em Proceedings of the Thirty-Ninth Annual ACM Symposium on Theory of Computing}, pages 75--84.

\bibitem[Opsahl and Panzarasa, 2009]{Opsahl:2009}
Opsahl, T. and Panzarasa, P. (2009).
\newblock Clustering in weighted networks.
\newblock {\em Social Networks}, 31(2):155--163.

\bibitem[Panzarasa et~al., 2009]{Panzarasa:2009}
Panzarasa, P., Opsahl, T., and Carley, K.~M. (2009).
\newblock Patterns and dynamics of users' behavior and interaction: Network analysis of an online community.
\newblock {\em Journal of the American Society for Information Science and Technology}, 60(5):911--932.

\bibitem[Qin et~al., 2017]{Qin:2017}
Qin, Z., Yu, T., Yang, Y., Khalil, I., Xiao, X., and Ren, K. (2017).
\newblock Generating synthetic decentralized social graphs with local differential privacy.
\newblock In {\em Proceedings of the 2017 ACM SIGSAC Conference on Computer and Communications Security}, pages 425--438.

\bibitem[Robins et~al., 2009]{Robins.et.al.2009}
Robins, G., Pattison, P., and Wang, P. (2009).
\newblock Closure, connectivity and degree distributions: Exponential random graph (p*) models for directed social networks.
\newblock {\em Social Networks}, 31(2):105--117.

\bibitem[Shao et~al., 2021]{Shao:2021}
Shao, M., Zhang, Y., Wang, Q., Zhang, Y., Luo, J., and Yan, T. (2021).
\newblock L-2 regularized maximum likelihood for $\beta$-model in large and sparse networks.
\newblock {\em arXiv preprint arXiv:2110.11856}.

\bibitem[Wasserman and Zhou, 2010]{Wasserman:Zhou:2010}
Wasserman, L. and Zhou, S. (2010).
\newblock A statistical framework for differential privacy.
\newblock {\em Journal of the American Statistical Association}, 105(489):375--389.

\bibitem[Yan, 2021]{Yan:2021}
Yan, T. (2021).
\newblock Directed networks with a differentially private bi-degree sequence.
\newblock {\em Statistica Sinica}, 31(4):2031--2050.

\bibitem[Yan, 2025]{Yan:2023}
Yan, T. (2025).
\newblock Differentially private analysis of networks with covariates via a generalized $\beta$-model.
\newblock {\em Science China Mathematics}, 68(10):2469–2500.

\bibitem[Yan et~al., 2016]{Yan:Leng:Zhu:2016}
Yan, T., Leng, C., and Zhu, J. (2016).
\newblock Asymptotics in directed exponential random graph models with an increasing bi-degree sequence.
\newblock {\em The Annals of Statistics}, 44(1):31--57.

\bibitem[Zhang and Chen, 2013]{Zhang:Chen:2013}
Zhang, J. and Chen, Y. (2013).
\newblock Sampling for conditional inference on network data.
\newblock {\em Journal of the American Statistical Association}, 108(504):1295--1307.

\bibitem[Zhen et~al., 2024]{Wang:2024}
Zhen, Y., Xu, S., and Wang, J. (2024).
\newblock Consistent community detection in multi-layer networks with heterogeneous differential privacy.
\newblock {\em arXiv preprint arXiv:2406.14772}.

\end{thebibliography}

\bibliographystyle{apalike} 


\end{document}